\documentclass[12pt, notitlepage]{amsart}
\usepackage{latexsym, amsfonts, amsmath, amssymb, amsthm, cite, tabls, paralist}

\newtheorem{theorem}{Theorem}[section]
\newtheorem{lemma}[theorem]{Lemma}
\newtheorem{proposition}[theorem]{Proposition}
\newtheorem{hypothesis}[theorem]{Hypothesis}

\theoremstyle{remark}
\newtheorem{remark}[theorem]{Remark}

\theoremstyle{definition}
\newtheorem{definition}[theorem]{Definition}

\theoremstyle{remark}

\newtheorem{case}{Case}

\numberwithin{equation}{section}

\newcommand{\Z}{\mathbb{Z}}

\newcommand{\R}{\mathbb{R}}

\newcommand{\C}{\mathbb{C}}
\newcommand{\E}{\mathbb{E}}

\newcommand{\CA}{\mathcal{A}}
\newcommand{\CB}{\mathcal{B}}
\newcommand{\CC}{\mathcal{C}}

\newcommand{\CH}{\mathcal{H}}
\newcommand{\CL}{\mathcal{L}}
\newcommand{\CQ}{\mathcal{Q}}
\newcommand{\FS}{\mathfrak{S}}
\newcommand{\Bohr}{\text{Bohr}}
\newcommand{\1}{\mathbf{1}}

\newcommand{\Prime}{\mathbb{P}}

\usepackage{graphicx}
\usepackage{mathrsfs}
\usepackage{varioref}

\begin{document}

\title{Vinogradov's three primes theorem with almost twin primes}

\author{Kaisa Matom\"aki}
\address{Department of Mathematics and Statistics\\ University of Turku \\
20014 Turku \\ Finland}
\email{ksmato@utu.fi}
\thanks{KM was supported by Academy of Finland grant no. 137883, 138522 and 285894.}

\author{Xuancheng Shao}
\address{Mathematical Institute\\ Radcliffe Observatory Quarter\\ Woodstock Road\\ Oxford OX2 6GG \\ United Kingdom}
\email{Xuancheng.Shao@maths.ox.ac.uk}
\thanks{XS is supported by a Glasstone Research Fellowship.}

\begin{abstract}
In this paper we prove two results concerning Vinogradov's three primes theorem with primes that can be called almost twin primes. First, for any $m$, every sufficiently large odd integer $N$ can be written as a sum of three primes $p_1, p_2$ and $p_3$ such that, for each $i \in \{1,2,3\}$, the interval $[p_i, p_i + H]$ contains at least $m$ primes, for some $H = H(m)$. Second, every sufficiently large integer $N \equiv 3 \pmod{6}$ can be written as a sum of three primes $p_1, p_2$ and $p_3$ such that, for each $i \in \{1,2,3\}$, $p_i + 2$ has at most two prime factors.
\end{abstract}

\maketitle

\section{Introduction}

The Hardy-Littlewood prime tuples conjecture says that, for any admissible set of $k$ integers $\CH = \{h_1, \cdots, h_k\}$, there are infinitely many values of $n$ such that $n + h_1, \cdots, n + h_k$ are all prime. Here $\CH$ is said to be admissible if it misses at least one residue class modulo $p$ for every prime $p$. In particular, the twin prime conjecture is the special case when $\CH = \{0,2\}$.

Using an elaboration of the linear sieve method, Chen \cite{Chen} proved that there are infinitely many primes $p$ such that $p+2$ is the product of at most two primes (this property is traditionally denoted by $p + 2 = P_2$). If one insists on prime values, it is only recently that Zhang \cite{Zhang}, and subsequently Maynard \cite{MaynardI}, made the breakthrough showing that there are infinitely many values of $n$ for which at least two of $n + h_1, \cdots, n + h_k$ are prime, provided that $k$ is large enough but fixed. Indeed, Maynard's argument shows that one can find $m$ primes among $n + h_1, \cdots, n + h_k$ for any $m$, provided that $k$ is large enough in terms of $m$. This result was proved independently by Tao in an unpublished work. We refer the reader to the excellent survey article \cite{Granville-survey} for the main ideas behind these works.

Since the introduction of the Hardy-Littlewood circle method, there have been a flurry of results about solving linear equations in prime variables, by analyzing exponential sums over primes. In 1937, Vinogradov showed that all sufficiently large odd positive integers can be written as a sum of three primes. This establishes the ternary version of the Goldbach conjecture. In this paper, we prove the analogous statement for the special types of almost twin primes mentioned above.

\begin{theorem}
\label{th:GoldbachMaynard}
For any positive integer $m$, there exist positive constants $H = H(m)$ and $N_0 = N_0(m)$ such that every odd integer $N \geq N_0$ can be written in the form $N = p_1 + p_2 + p_3$, where, for $i = 1, 2, 3$, $p_i$ are primes such that the interval $[p_i, p_i + H]$ contains at least $m$ primes.
\end{theorem}

In view of recent work of Helfgott \cite{Helfgott}, one can in fact take $N_0 = 7$ above (after possibly increasing $H$).

\begin{theorem}
\label{th:GoldbachChen}
Every large enough integer $N \equiv 3 \pmod{6}$ can be written in the form $N = p_1 + p_2 + p_3$, where, for $i = 1, 2, 3$, $p_i$ are primes such that $p_i + 2$ is a product of at most two primes.
\end{theorem}

Related problems have been considered before. Green and Tao \cite{GreenTaoRestr} showed that there are infinitely many three-term arithmetic progressions in the almost twin primes considered in Theorem~\ref{th:GoldbachChen}, and this has been generalized in \cite{kAP-Chen} to handle $k$-term progressions for any fixed $k$. See \cite{kAP-Maynard} for analogous results for the almost twin primes considered in Theorem~\ref{th:GoldbachMaynard}. As we will discuss in the next section, since the equation $N = p_1 + p_2 + p_3$ is not translation-invariant, for subsets of the primes the ternary Goldbach problem involves additional complications compared to the problem of finding three-term arithmetic progressions. For the ternary Goldbach problem, Matom\"aki \cite{MatomakiB-V} previously showed that $N = p_1 + p_2 + p_3$ is solvable in primes with $p_1+2 = P_2$, $p_2 + 2 = P_2'$, and $p_3 + 2 = P_7$.

It is worth mentioning that a vast generalization of Vinogradov's theorem has been proved by Green and Tao \cite{GT-linear-equations}, with a crucial ingredient from the work of Green, Tao, and Ziegler \cite{GTZ}. They introduced the concept of higher order Fourier analysis, which allows one to handle all linear systems of finite complexity (that excludes the twin prime or the binary Goldbach case). We plan to return to a generalization of Theorem \ref{th:GoldbachMaynard} in this direction in a future work.

\subsection*{Acknowledgements}  
This work started when both authors were visiting CRM in Montreal during the analytic part of the thematic year in number theory in Fall 2014, whose hospitality is greatly appreciated. The authors are grateful to Joni Ter{\"a}v{\"a}inen for pointing out a few mistakes in an earlier draft, and to the anonymous referee for valuable suggestions.

\section{Outline of proof}
\label{sec:outline}
In this section we describe the main ingredients in the proofs of Theorems~\ref{th:GoldbachMaynard} and~\ref{th:GoldbachChen}. The general strategy for proving both theorems follows closely the transference principle initiated in \cite{Green3AP}. Let $f$ be the (weighted) indicator function of the considered subset of the primes, and let $\nu$ be a sieve majorant so that $f \leq \nu$ and that $f$ has positive density in $\nu$. The Fourier analytic transference principle in \cite{Green3AP} produces a dense model $\widetilde{f}$ of $f$, such that $0 \leq \widetilde{f} \leq 1$ and that $\widetilde{f}$ has positive average. Moreover,
\begin{equation}\label{eq:f-ftilde} 
\sum_{\substack{1 \leq n_1,n_2,n_3 \leq N \\ n_1 + n_2 + n_3 = N}} f(n_1)f(n_2)f(n_3) \approx \sum_{\substack{1 \leq n_1,n_2,n_3 \leq N \\ n_1 + n_2 + n_3 = N}} \widetilde{f}(n_1) \widetilde{f} (n_2) \widetilde{f}(n_3). 
\end{equation}
If we are instead looking for solutions of a homogeneous linear equation such as $n_1 + n_2 = 2n_3$, then the right hand side above is bounded from below by Roth's theorem. In this way one can find arithmetic progressions in subsets of primes~\cite{GreenTaoRestr, kAP-Chen, kAP-Maynard}. In our current case, the right hand side above could vanish if, for example, $\widetilde{f}$ is supported on $[1,N/4]$ or if, writing $\Vert x \Vert$ for the distance from the nearest integer, we had $\Vert \sqrt{2}N \Vert > 3/10$ and $\widetilde{f}$ is supported on numbers $n$ for which $\Vert \sqrt{2} n \Vert < 1/10$.

To get around this issue, we need to know more about the structure of $\widetilde{f}$. Examining the proof of the transference principle, one may observe that $\widetilde{f}$ is the convolution of $f$ with a Bohr set. If we ensure that $\widetilde{f}$ is bounded below pointwise, then the right hand side of \eqref{eq:f-ftilde} is certainly bounded below as well. This pointwise lower bound translates to the requirement that primes from the considered subset can be found in Bohr sets.

\subsection{Smooth Bohr cutoff}

Given a cyclic group $G = \Z/N\Z$, a subset $\Omega \subseteq G$ and $\eta \in (0, 1/2]$, define the Bohr set
\[ B = \Bohr(\Omega, \eta) = \{ n \in G: \|\xi n / N\| \leq \eta\text{ for all }\xi \in \Omega\}. \]
For technical reasons, it is more convenient to study a smooth version of $1_B$, whose Fourier spectrum has bounded size.

For $\eta \in (0, 1/2]$ and a positive integer $D$, let $S_{D, \eta}^+(x)\colon \R/\Z \to [0, 2]$ be the Selberg polynomial of degree $D$ that majorizes the interval $[-\eta, \eta]$.  The definition can be bound in~\cite[Chapter 1, formula $21^+$]{Montgomery10Lect} and is given in~\eqref{eq:SD+def} below. The Selberg polynomial has a Fourier expansion
\[
S_{D, \eta}^+(x) = \sum_{|k| \leq D} \widehat{S}_{D, \eta}^+(k) e(kx)
\]
with $|\widehat{S}_{D, \eta}^+(k)| \leq \frac{1}{D+1}+\min\{2\eta, 1/|k|\}$ by \cite[Chapter 1, formula (22)]{Montgomery10Lect}.
\begin{definition}[Smooth Bohr cutoff]\label{def:smooth-Bohr}
Given a cyclic group $G = \Z/N\Z$, a subset $\Omega \subseteq G$ and $\eta \in (0, 1/2]$, let $D = \lceil 4/\eta \rceil^{2|\Omega|}$ and define the smooth Bohr cutoff $\chi = \chi_{\Omega, \eta} \colon G \to \mathbb{R}_{\geq 0}$ by
\[
\chi(n) := \prod_{\xi \in \Omega} S^+_{D, \eta}(\xi n / N).
\]
\end{definition}
Note that since $S^+_{D, \eta}(x)$ is a majorant of $1_{\Vert x \Vert \leq \eta}(x)$, we have the lower bound $\chi(n) \geq 1$ for $n \in \Bohr(\Omega, \eta)$.

\begin{remark}
Using the Selberg polynomials $S^+_{D, \eta}$ is not essential here --- one could replace them for instance by the function $(\cos \pi x)^{D}$ for some large even $D$ depending on $\eta$ and $|\Omega|$. This way $\chi(n)$ would no longer be at least $1$ in the Bohr set, but one could easily prove good enough variants of the lemmas we need.
\end{remark}

\subsection{A transference type result}

Let $G = \Z/N\Z$. We use the standard notation $\E_{n \in G}$ to denote the average $N^{-1}\sum_{n \in G}$. For a function $f: G \rightarrow \C$ , its Fourier transform is defined by
\[ 
\widehat{f}(\xi) = \E_{n \in G} f(n) e\left( -\frac{\xi n}{N} \right),
\]
and its $L^1$-norm is defined by
\[ \|f\|_1 = \E_{n \in G} |f(n)|. \]
For two functions $f, g: G\rightarrow \C$, their convolution is defined by
\[ f*g(t) = \E_{n \in G} f(n) g(t-n). \]
In Section~\ref{sec:transfer} we prove the following transference type result. It says that we can handle a non-homogeneous linear equation if we have some additional hypotheses about averages in Bohr sets. 

\begin{theorem}\label{thm:transfer}
Let $G = \Z/N\Z$ for some large $N$, and let $f_1: G \rightarrow \R_{\geq 0}$ be a function. Let $K \geq 1$ and $\delta > 0$ be parameters. There exists a Bohr cutoff $\chi = \chi_{\Omega,\eta}$ (depending on $f_1$) with $|\Omega| \ll_{K,\delta} 1$, $1 \in \Omega$, and $\eta = \eta(K,\delta) \in (0,0.05)$, such that the following statement holds. 
Let $f_2,f_3: G \rightarrow \R_{\geq 0}$ be functions satisfying
\begin{equation}\label{eq:rich-Bohr} 
f_i*\chi(t) \geq \delta \|\chi\|_1,
\end{equation}
for every $t \in [N/4,N/2)$ and $i \in \{2,3\}$. Suppose that
\begin{equation}
\label{eq:f1avlarge}
\sum_{0.1N \leq n \leq 0.4N} f_1(n) \geq \delta N,
\end{equation}
and that
\begin{equation}
\label{eq:5/2norm}
\sum_{\xi \in G} |\widehat{f_i}(\xi)|^{5/2} \leq K 
\end{equation}
for every $i \in \{1,2,3\}$. Then $f_1*f_2*f_3(N) \geq \delta^3/200$.
\end{theorem}

The artificial requirement $1 \in \Omega$ and the assumption that \eqref{eq:rich-Bohr} holds only for $t \in [N/4, N/2)$ come from the way Theorem~\ref{thm:transfer} will be applied. In order to avoid wrapping around issues, we will apply Theorem~\ref{thm:transfer} with each $f_i$ supported on $[N/4,N/2)$. If $1\in \Omega$ and $\eta < 0.1$, then $B(\Omega, \eta) \subset (-0.1N, 0.1N)$, so that \eqref{eq:rich-Bohr} can be expected to hold when $t \in [N/4, N/2)$.

We will see that the condition~\eqref{eq:5/2norm} for the types of almost twin primes we consider follows easily from the work of Green and Tao\cite{GreenTaoRestr}.

\subsection{Almost twin primes in Bohr sets}

To apply Theorem~\ref{thm:transfer} to prove Theorems~\ref{th:GoldbachMaynard} and~\ref{th:GoldbachChen} in Section~\ref{sec:proofsOfMainTheorems}, we need to verify the hypothesis \eqref{eq:rich-Bohr} for the indicator functions of the types of almost twin primes we consider. This is achieved in Theorems~\ref{th:MaynPrimesinBohrSets} and \ref{th:ChenPrimesinBohrSets}, in statements of which we use the following definition.

\begin{definition}\label{def:Fourier-complexity}
For a function $\chi: \Z \rightarrow \C$, we say that it has Fourier complexity at most $M$ if $\chi$ can be written as a linear combination of at most $M$ exponential phases:
\[ \chi(n) = \sum_{i=1}^M b_i e(\alpha_i n), \]
for some $|b_i| \leq M$, and $\alpha_i \in \R/\Z$. 
\end{definition}
Note that since we do not request $b_i$ to be non-zero, if $\chi$ is of Fourier complexity at most $M$, then it is of Fourier complexity at most $M'$ for any $M' \geq M$. Note also that the smooth Bohr cutoff $\chi_{\Omega,\eta}$ in Definition~\ref{def:smooth-Bohr} (extended to $\Z$ in the obvious manner) has Fourier complexity at most $O_{|\Omega|,\eta}(1)$.

\begin{theorem}
\label{th:MaynPrimesinBohrSets}
For any positive integer $m$, there exist a positive integer $k = k(m)$ and positive constants $\delta_0 = \delta_0(m)$ and $\rho = \rho(m)$ such that the following holds. Let $\chi: \Z \rightarrow \R_{\geq 0}$ be a function with Fourier complexity at most $M$ for some $M \geq 1$, and let $\varepsilon > 0$ be given. Let $W = \prod_{p \leq w} p$ with $w$ large enough in terms of $m, M$ and $\varepsilon$, and let $(b, W) = 1$. There exist non-zero distinct integers $h_1, \dotsc, h_{k-1} = O_{m, M, \varepsilon}(1)$ with $h_j$ positive for $j= 1, \dotsc, m-1$, and a positive integer $N_0 = N_0(m, M, \varepsilon, w)$ such that, for every $N \geq N_0$ and $|t| \leq 5 N$,
\[ 
\sum_{\substack{ N \leq n < 2N \\ Wn+b \in \mathbb{P} \\ Wn+b+Wh_i \in \mathbb{P} \text{ for $i = 1, \dotsc, m-1$} \\ p \mid \prod_{i = m}^{k-1} (Wn+b+Wh_i) \implies p \geq N^{\rho}}} \chi(t-n) \geq \delta_0 \frac{1}{(\log N)^k} \frac{W^k}{\varphi(W)^k} \left(\sum_{N \leq n < 2N} \chi(t-n) - \frac{N}{w^{1/3}} + O_m(\varepsilon N)\right).
\]
\end{theorem}

\begin{theorem}\label{th:ChenPrimesinBohrSets}
There exists a positive constant $\delta_1$ such that the following holds. Let $\chi: \Z \rightarrow \R_{\geq 0}$ be a function with Fourier complexity at most $M$ for some $M \geq 1$. Let $W = \prod_{p \leq w} p$ with $w$ large enough in terms of $M$, and let $(b, W) = 1$. There exists a positive constant $N_0 = N_0(M, w)$ such that, for every $N \geq N_0$ and $|t| \leq 5 N$,
\[ \sum_{\substack{N \leq n < 2N \\ Wn + b \in \Prime \\ Wn + b + 2 = P_2 \\ p \mid Wn + b + 2 \implies p \geq N^{1/100}}} \chi(t - n) \geq \delta_1 \frac{1}{(\log N)^2} \frac{W^2}{\varphi(W)^2} \left( \sum_{N \leq n < 2N} \chi(t - n) - \frac{N}{w^{1/3}} \right). \]
\end{theorem}

Let us briefly discuss the proofs of these results. In Section~\ref{se:MaynardChenTheorems} we shall state the results of Maynard and Chen saying that one can find almost twin primes in sets that are equidistributed in arithmetic progressions in certain precise senses. Bohr sets in general are not equidistributed but we will in Section~\ref{se:technRed} show that it is enough to show variants of Theorems~\ref{th:MaynPrimesinBohrSets} and~\ref{th:ChenPrimesinBohrSets} that are more apt for applications of Maynard's and Chen's theorems. Then in Sections~\ref{sec:MaynApplication} and~\ref{sec:ChenApplication} we shall prove these variants using the Fourier expansion of the smooth Bohr cutoff discussed in Section~\ref{sec:smoothBohrCutoff} as well as exponential sum estimates which we will state in Section~\ref{sec:expsumest}.

\section{Smooth Bohr cutoff and its Fourier expansion}
\label{sec:smoothBohrCutoff}
In this section we discuss a few basic properties of the Bohr cutoff $\chi = \chi_{\Omega,\eta}$ from Definition~\ref{def:smooth-Bohr}.

\begin{lemma}
\label{le:BohrCutOffProp}
Given a cyclic group $G = \Z/N\Z$, a subset $\Omega \subseteq G$ and $\eta \in (0, 1/2]$, the smooth Bohr cutoff $\chi = \chi_{\Omega, \eta}$ has the following properties.
\begin{enumerate}
\item 
We have the lower bound
\[
\|\chi\|_1 \geq (\eta/2)^{|\Omega|} ;
\]
\item If $n \notin \Bohr(\Omega, 2\eta)$, then
\[
|\chi(n)| \leq (\eta^2/8)^{|\Omega|}.
\]
\end{enumerate}
\end{lemma}
\begin{proof}
Part (i) follows from the observation that $\chi(n) \geq 1$ when $n \in \Bohr(\Omega, \eta)$, together with the lower bound $|\Bohr(\Omega, \eta)| \geq (\eta/2)^{|\Omega|} N$ from a standard pigeon-holing argument (see e.g. \cite[Lemma 4.20]{TV10}). For part (ii) we can clearly assume that $\eta \leq 1/4$. Let us first give the precise definition of $S_{D, \eta}^+(x)$. For an integer $K \geq 1$, write $\Delta_K(x)$ for the Fej\'er kernel
\[
\Delta_K(x) := \sum_{|k| \leq K} \left(1-\frac{|k|}{K}\right)e(kx) = \frac{1}{K}\left(\frac{\sin \pi K x}{\sin \pi x}\right)^2.
\]
Then Vaaler's polynomial $V_D(x)$ is defined as the trigonometric polynomial of degree $D$ with
\[
\begin{split}
V_D(x) &:= \frac{1}{D+1} \sum_{k =1}^D \left(\frac{k}{D+1}-\frac{1}{2}\right) \Delta_{D+1}\left(x-\frac{k}{D+1}\right) \\
&\qquad + \frac{1}{2\pi (D+1)} \sin 2\pi(D+1)x - \frac{1}{2\pi} \Delta_{D+1}(x) \sin 2\pi x.
\end{split}
\]
Finally
\begin{equation}
\label{eq:SD+def}
S_{D, \eta}^+(x) := 2\eta + V_D(x-\eta) + V_D(-x-\eta) + \frac{1}{2D+2}\left(\Delta_{D+1}(x-\eta) + \Delta_{D+1}(-x-\eta)\right).
\end{equation}
Note that, writing $s(x)$ for the sawtooth function (so that $s(x) = \{x\}-1/2$ if $x\notin\Z$ and $s(x) = 0$ if $x \in \Z$),
\[
1_{\Vert x \Vert \leq \eta} (x) = 2\eta + s(x-\eta) + s(-x-\eta),
\]
except when $x = \eta$ or $x = -\eta$. By a result of Vaaler~\cite[Theorem 18]{Vaaler}, we know that, for any $x$,
\[
|V_D(x) - s(x)| \leq \frac{1}{2D+2} \Delta_{D+1}(x).
\]
Hence
\[
\begin{split}
|S_{D, \eta}^+(x) - 1_{\Vert x \Vert \leq \eta} (x)| &\leq  \frac{2}{2D+2}\left(\Delta_{D+1}(x-\eta) + \Delta_{D+1}(-x-\eta)\right) \\
& \leq \frac{1}{(D+1)^2} \left(\frac{1}{(\sin \pi \Vert x-\eta \Vert)^2} + \frac{1}{(\sin \pi \Vert-x-\eta\Vert)^2}\right).
\end{split}
\]
If $\Vert x \Vert \geq 2\eta$ then we get
\[
|S_{D, \eta}^+(x)| \leq \frac{2}{(D+1)^2}  \cdot \frac{1}{(\sin \pi \eta)^2} \leq \frac{2}{(D+1)^2}  \cdot \frac{1}{(2 \eta)^2} \leq \frac{1}{\eta^2 D^2}.
\]
Now, if $n \notin \Bohr(\Omega, 2\eta)$ then $\Vert \xi_0 n / N \Vert \geq 2\eta$ for some $\xi_0 \in \Omega$.
Thus,
\[
|\chi(n)| = \left| S_{D,\eta}^+(\xi_0 n / N) \right| \prod_{\xi \in \Omega \setminus \{\xi_0\}} \left| S^+_{D, \eta}(\xi n / N)\right|  \leq \frac{2^{|\Omega|}}{\eta^2D^2}.
\]
The conclusion then follows by our choice $D = \lceil 4/\eta \rceil^{2|\Omega|}$.
\end{proof}

%\subsection{Decomposition into equidistributed functions}

%Since Bohr sets (and in general functions with bounded Fourier complexity) are not equidistributed in arithmetic progressions, we cannot apply Maynard's theorem to the situation in Proposition~\ref{prop:AlmMaynPrimesinBohrSets} directly, but we need to be careful with our choice of sequence $\omega_n$ to which we apply Maynard's theorem. Let us now do some initial preparations to find out the problematic moduli for which the equidistribution fails.

%\begin{definition}\label{def:R-equi}
%Let $R$ be a positive integer and $x \geq 2$ be real. A phase $\alpha \in \R/\Z$ is called $R$-irrational at scale $x$ if $\| r\alpha \| > R/x$ for any $1 \leq r \leq R$. A function $\chi: \Z \rightarrow \C$ with Fourier complexity at most $M$ is called $R$-equidistributed at scale $x$ if $f$ can be written as
%\[ \chi(n) = \sum_{i=1}^m c_i e(\alpha_i n), \]
%for some $m \leq M$, $|c_i| \leq M$, and $\alpha_i \in \R/\Z$ which is either $0$ or $R$-irrational at scale $x$.
%\end{definition}
%
%
%\begin{lemma}\label{le:chi-Fourier}
%Let $R$ be a positive integer and $x \geq 2$ be real. Let $\chi: \Z \rightarrow \C$ be a function with Fourier complexity at most $M$. There exists a positive integer $Q \leq R^M$ such that the scaled function $\widetilde{\chi}: \Z \rightarrow \C$ defined by $\widetilde{\chi}(n) = \chi(Q^2n)$ is $R$-equidistributed at scale $x$.
%\end{lemma} 

The following lemma gives the Fourier expansion of a function of bounded Fourier complexity in a convenient form. In particular it allows us to separate the phases giving ``major arc'' contribution from those giving ``minor arc'' contribution.

\begin{lemma}\label{le:chi-Fourier}
Let $A, M \geq 1$, and let $B = A (3M)^M$.  Let $\chi: \Z \rightarrow \C$ be a function with Fourier complexity at most $M$, and let $W$ be a positive integer. Then for any large $N$ we may write
\[ \chi(n) = \sum_{i=1}^M b_i e\left( \left( W\frac{a_i}{q_i} + \beta_i \right) n \right) \]
for some $|b_i| \leq M$, $0 \leq a_i < q_i \leq N/(\log N)^{100B}$, $(a_i,q_i) = 1$, and $|\beta_i| \leq W(\log N)^{100B}/(q_iN)$. Moreover, there exists a positive integer $Q \leq (\log N)^B$ such that, for each $1 \leq i \leq M$, either $q_i \mid Q$ or $q_i/(q_i,Q^2) > (\log N)^A$.
\end{lemma}

\begin{proof}
By the definition of Fourier complexity in Definition~\ref{def:Fourier-complexity}, we may write
\[ \chi(n) = \sum_{i =1}^M b_i e(\alpha_i n), \]
for some $|b_i| \leq M$ and $\alpha_i \in \mathbb{R} / \mathbb{Z}$.  By the Dirichlet approximation theorem, for each $1 \leq i \leq M$, there exist integers $q_i \in [1, N/(\log N)^{100B}]$ and $a_i$ such that $(a_i, q_i)= 1$ and 
\[ \left|\frac{\alpha_i}{W}-\frac{a_i}{q_i} \right| \leq \frac{(\log N)^{100B}}{q_iN}.  \]
This gives the desired Fourier expansion of $\chi$, apart from the existence of $Q$ mentioned in the last sentence of the statement.

To define $Q$, let $\CQ = \{q_1, \dotsc, q_M\}$. Take $Q_0 = 1$ and  for $i \geq 0$ define
\[
Q_{i+1} = \prod_{\substack{q \in \CQ \\ \frac{q}{(q, Q_i^2)} \leq (\log N)^{A}}} q.
\]
There is some $I \leq |\mathcal{Q}| = M$ such that $Q_{I+1} = Q_I$. We claim that $Q = Q_I$ satisfies the desired properties. Indeed, for $q \in \CQ$, if $q \nmid Q$, then $q \nmid Q_{I+1}$ so that $q/(q, Q_I^2) > (\log N)^A$ by the definition of $Q_{I+1}$. Furthermore, it is easy to see from the construction that 
\[ Q_{i+1} \leq (Q_i^2 (\log N)^A)^M. \]
Thus a simple induction reveals that $Q_i \leq (\log N)^{A \cdot 3^{i} M^i}$, so that $Q \leq (\log N)^B$.
\end{proof}

This lemma can be thought of as a very special case of the general factorisation theorem for nilsequences \cite[Theorem 1.19]{GT-nilsequence}.

%Recalling the definition of $\chi$ and the Fourier expansion of $S^+_{D, \eta}(x)$, we obtain
%
%\begin{fact}
%\label{fa:chiFourierExp}
%The function $\chi(n)$ has the Fourier expansion
%\[
%\chi(n) %= \sum_{(s_1, \dotsc, s_d) \in S} \hat{h}(s_1) \dotsm \hat{h}(s_d) e(L(s)n) 
%= \sum_{s \in S} b_s e\left(\left(\frac{Wa(s)}{q(s)} + \beta(s)\right) n\right),
%\] 
%where, for each $s \in S$, $q(s) \in [1, N/(\log N)^{100B}]$, $(a(s), q(s))= 1$, $|\beta(s)| < \frac{(\log N)^{100B}W}{q(s)N}$ and $|b_s| \leq (4\eta)^d$.
%\end{fact}

%From the Fourier expansion of $\chi(n)$ one sees that $\chi(n)$ typically fails to be equidistributed $\pmod{q(\alpha)}$ for $\alpha \in \Omega$. This is most serious for small moduli, and we will restrict $n$ to residue classes modulo $Q_1$ which is the product of ``small moduli'' $q(\alpha)$. But after such restriction we might find some further problematic moduli for which $q(\alpha)/(q(\alpha), Q_1)$ is small since now $n$ runs essentially $k + rQ_1$. To avoid such problems, we make a recursive definition of $Q$ modulo which we shall split into residue classes.

\section{The transference type result}
\label{sec:transfer}

In this section we prove Theorem~\ref{thm:transfer}. Let $\eta, \varepsilon > 0$ be small enough depending on $K$ and $\delta$, and take 
\[ \Omega = \{ \xi \in G: |\widehat{f_1}(\xi)| \geq \varepsilon \} \cup \{1\}. \]
By~\eqref{eq:5/2norm}, we have $|\Omega| \leq \varepsilon^{-5/2} K + 1$. Let $\chi = \chi_{\Omega,\eta}$ be the smooth Bohr cutoff from Definition~\ref{def:smooth-Bohr}. For $i \in \{2,3\}$, define $g_i, h_i: G \rightarrow \R$ by setting
\[ g_i = \frac{1}{\|\chi\|_1} f_i*\chi, \ \ h_i = f_i - g_i. \]
Hence
\begin{equation}
\label{eq:ghathhat}
\widehat{g_i} = \frac{1}{\|\chi\|_1} \widehat{f_i} \cdot \widehat{\chi} \quad \text{and} \quad \widehat{h_i} = \widehat{f_i} \left( 1 - \frac{\widehat{\chi}}{\|\chi\|_1} \right). 
\end{equation}
In particular, using the trivial bound $|\widehat{\chi}(\xi)| \leq \|\chi\|_1$ we obtain
\begin{equation}
\label{eq:5/2normgh}
\sum_{\xi \in G} |\widehat{g_i}(\xi)|^{5/2} \leq K \quad \text{and} \quad \sum_{\xi \in G} |\widehat{h_i}(\xi)|^{5/2}  \leq 2^{5/2} K. 
\end{equation}
We write
\[
f_1 \ast f_2 \ast f_3(N) = f_1 \ast g_2 \ast g_3(N) + f_1 \ast g_2 \ast h_3(N) + f_1 \ast h_2 \ast g_3(N) + f_1 \ast h_2 \ast h_3(N).
\]
By the assumption \eqref{eq:rich-Bohr} we have, for $i \in \{1, 2\}$ the pointwise lower bound $g_i(t) \geq \delta$ for all $t \in [N/4, N/2)$. Thus
\[ 
f_1 * g_2 * g_3(N) \geq \frac{1}{N^2} \sum_{n_1} f_1(n_1) \sum_{\substack{N/4 \leq n_2,n_3 < N/2  \\ n_1 + n_2 + n_3 = N}} \delta^2 \geq \frac{\delta^2}{100N} \sum_{0.1N \leq n_1 \leq 0.4N} f_1(n_1) \geq \frac{1}{100} \delta^3
\]
by the assumption~\eqref{eq:f1avlarge}.

To conclude the proof it remains to show that
\begin{equation}
\label{eq:fhhupbound}
\left| f_1 * h_2 * h_3(N) \right| \leq \frac{1}{1000} \delta^3,
\end{equation}
and the same bound with either $h_2$ replaced by $g_2$ or $h_3$ replaced by $g_3$. We have
\begin{equation}
\label{eq:f1h2h3Fourier}
\left| f_1*h_2*h_3(N) \right| \leq \sum_{\xi \in G}  |\widehat{f_1}(\xi) \widehat{h_2}(\xi) \widehat{h_3}(\xi)|.
\end{equation}
First we bound the contribution of summands with $\xi \notin \Omega$. By the definition of $\Omega$ we have $|\widehat{f_1}(\xi)| < \varepsilon$ for $\xi \notin \Omega$. Thus
\[ \sum_{\xi \in G \setminus \Omega}  |\widehat{f_1}(\xi) \widehat{h_2}(\xi) \widehat{h_3}(\xi)| < \varepsilon^{1/2} \sum_{\xi \in G}  |\widehat{f_1}(\xi)|^{1/2}  |\widehat{h_2}(\xi) \widehat{h_3}(\xi)|. \]
By H\"older's inequality, this is bounded by
\[ \varepsilon^{1/2} \left(\sum_{\xi \in G} |\widehat{f_1}(\xi)|^{5/2}\right)^{1/5} \left(\sum_{\xi \in G} |\widehat{h_2}(\xi)|^{5/2}\right)^{2/5} \left(\sum_{\xi \in G} |\widehat{h_3}(\xi)|^{5/2}\right)^{2/5} \leq 4K \varepsilon^{1/2}, \]
by~\eqref{eq:5/2norm} and~\eqref{eq:5/2normgh}. This is acceptable if $\varepsilon$ is small enough. To bound the contribution to the right hand side of~\eqref{eq:f1h2h3Fourier} of summands with $\xi \in \Omega$, it suffices to show that $|\widehat{h_2}(\xi)| \leq 30\eta K^{2/5}$ for $\xi \in \Omega$ (the rest of the argument follows just as above). Since, by~\eqref{eq:5/2norm}, $|\widehat{f_2}(\xi)| \leq K^{2/5}$, by~\eqref{eq:ghathhat} it suffices to show that
\[ \left| 1 - \frac{\widehat{\chi}(\xi)}{\|\chi\|_1} \right| \leq 30\eta \]
for $\xi \in \Omega$. We may write
\[ 1 - \frac{\widehat{\chi}(\xi)}{\|\chi\|_1}   = \frac{1}{N \|\chi\|_1} \sum_{n \in G} \chi(n) (1 - e(\xi n/N)) . \]
If $n \in \Bohr(\Omega, 2\eta)$, then $|1- e(\xi n/N)| \leq 20 \eta$. If $n \notin \Bohr(\Omega, 2\eta)$, then by Lemma~\ref{le:BohrCutOffProp} we have $|\chi(n)| \leq \eta \|\chi\|_1$.  Combining these together we obtain
\[ \left| 1 - \frac{\widehat{\chi}(\xi))}{\|\chi\|_1}  \right| \leq \frac{1}{N \|\chi\|_1} \left( 20\eta \sum_{n \in G} \chi(n) + \sum_{n \in G} 2\eta \|\chi\|_1 \right) \leq 30\eta, \]
as desired. This completes the proof of \eqref{eq:fhhupbound} and the cases where either $h_2$ is replaced by $g_2$ or $h_3$ is replaced by $g_3$ follow completely similarly. Hence Theorem~\ref{thm:transfer} follows.

\begin{remark}
Theorem~\ref{thm:transfer} in particular says that if, for a positive density subset of the primes, the ternary Goldbach does not hold for all large odd $N$, then there must be some sort of Bohr set obstruction (including, as special cases, local obstructions modulo primes), since the condition~\eqref{eq:5/2norm} holds in this case by the work of Green and Tao~\cite{GreenTaoRestr}. On the other hand, as mentioned in Section~\ref{sec:outline}, such obstructions may indeed prevent ternary Goldbach from holding.
\end{remark}

\begin{remark}
The condition \eqref{eq:rich-Bohr} should be compared with the usual hypotheses needed in carrying out the circle method. In a traditional application of the circle method, one requires the set to be equidistributed in Bohr sets so that the minor arc contributions are negligible, leading to an asymptotic formula for the number of solutions. In Theorem~\ref{thm:transfer}, with a weaker assumption \eqref{eq:rich-Bohr} about distribution in Bohr sets, we deduce a lower bound for the number of solutions (of the correct order of magnitude).
\end{remark}

\section{Proof of Theorems~\ref{th:GoldbachMaynard} and \ref{th:GoldbachChen} assuming Theorems~\ref{th:MaynPrimesinBohrSets} and~\ref{th:ChenPrimesinBohrSets}}
\label{sec:proofsOfMainTheorems}
In this section we deduce Theorems~\ref{th:GoldbachMaynard} and~\ref{th:GoldbachChen} from the transference principle, Theorems~\ref{th:MaynPrimesinBohrSets} and~\ref{th:ChenPrimesinBohrSets} and the work of Green and Tao~\cite{GreenTaoRestr}. Let us first record the consequence of\cite{GreenTaoRestr} we shall need. Here and later we call a set of linear forms $\CL = \{L_1, \dotsc, L_k\}$ admissible if they are distinct and $\prod_{i=1}^k L_i(n)$ has no fixed prime divisors. In this case we define the singular series 
\begin{equation}\label{eq:singseries}
\FS(\CL) = \prod_{p \in \mathbb{P}} \left(1-\frac{|\{n \in \mathbb{Z} / p\mathbb{Z} \colon p \mid L_1(n)\cdots L_k(n)\}|}{p}\right)\left(1-\frac{1}{p}\right)^{-k}. 
\end{equation}

%\begin{proposition}
%\label{prop:GreenTaoLpnorm}
%For any integers $k \geq m \geq 0$ there exists a positive constant $K = K(k)$ such that the following holds. Let $\mathcal{L} = \{L_1, \dotsc, L_k\}$ be an admissible set of $k$ linear functions $L_i(n) = a_i n + b_i$ with $|a_i|, |b_i| \leq N$. Write
%\[
%\begin{split}
%X = \{n \leq N \colon L_i(n) \in \mathbb{P} \text{ for } i = 1, \dotsc, m, \text{ and } p \mid \prod_{i = m+1}^{k} L_i(n) \implies p \geq N^{1/1000}\},
%\end{split}
%\]
%and
%\[
%\mathfrak{S} = \prod_{p \in \mathbb{P}} \left(1-\frac{|\{n \in \mathbb{Z} / p\mathbb{Z} \colon p \mid F(n)\}|}{p}\right)\left(1+\frac{1}{p}\right)^{k}.
%\]
%Let $f \colon \mathbb{Z} \to \mathbb{R}_{\geq 0}$ be defined by
%\[
%f(n) = 
%\begin{cases}
%(\log N)^k / \mathfrak{S} &\text{if $n \in X$;} \\
%0 & \text{otherwise.}
%\end{cases}
%\]
%Then
%\[
%\int_0^1 |\widehat{f}(\theta)|^{5/2} d\theta \leq K N^{3/2}.
%\]
%\end{proposition}

\begin{proposition}
\label{prop:GreenTaoLpnorm}
Let $\rho \in (0,1/2)$ be real and let $k \geq 1$ be an integer. Let $\mathcal{L} = \{L_1, \dotsc, L_k\}$ be an admissible set of $k$ linear functions $L_i(n) = a_i n + b_i$ with $|a_i|, |b_i| \leq N$. Write
\[
\begin{split}
X = \{n \leq N \colon p \mid \prod_{i = 1}^{k} L_i(n) \implies p \geq N^{\rho}\},
\end{split}
\]
and let $\FS(\CL)$ be defined as in \eqref{eq:singseries}.
Let $G = \Z/N\Z$ and let $f \colon G \to \mathbb{R}_{\geq 0}$ be such that
\[
f(n) \leq 
\begin{cases}
(\log N)^k / \mathfrak{S} &\text{if $n \in X$;} \\
0 & \text{otherwise.}
\end{cases}
\]
Here we naturally identified $G$ with $\{1,2,\dotsc ,N\}$. Then
\[
\sum_{\xi \in G} |\widehat{f}(\xi)|^{5/2} \leq K,
\]
for some positive constant $K = K(k, \rho)$.
\end{proposition}

\begin{proof}
Let $F = L_1L_2 \cdots L_k$, $R = N^{\rho/2}$, and let $\beta_R(n)$ be the enveloping sieve given by~\cite[Proposition 3.1]{GreenTaoRestr}, so that $\beta_R(n) \gg_{k,\rho} f(n)$. Applying~\cite[Proposition 4.2]{GreenTaoRestr} with $a_n = f(n)/\beta_R(n)$ if $\beta_R(n) \neq 0$ and $a_n = 0$ otherwise, we obtain that
\[
\left(\sum_{\xi \in G} |\widehat{f}(\xi)|^{5/2}\right)^{2/5} \ll_k  \left( \E_{n \leq N} a_n^2 \beta_R(n) \right)^{1/2} \ll_{k,\rho}  \left( \E_{n \leq N} \beta_R(n) \right)^{1/2} \ll_{k,\rho} 1, 
\]
where the last inequality follows from~\cite[Lemma 4.1]{GreenTaoRestr}.
\end{proof}

% \begin{proposition}
% \label{prop:GreenTaoLpnorm}
% For any integers $k \geq m \geq 1$ there exists a positive constant $K = K(k, m)$ such that the following holds. Let $W = \prod_{p \leq w} p$ for some $w \geq 1$. Assume that $(b, W) = 1$ and that $h_1, \dotsc, h_k$ are distinct integers satisfying $|h_j| \leq w/2$. Let $N \geq 1$ and
% \begin{equation}
% \label{eq:Xdef}
% \begin{split}
% X = \{n \leq N \colon Wn+b \in \mathbb{P}, & \quad Wn+b+Wh_i \in \mathbb{P} \text{ for } i = 1, \dotsc, m-1, \\
% &\text{ and } \quad p \mid \prod_{i = m}^{k-1} (Wn+b+Wh_i) \implies p \geq N^{1/1000}\},
% \end{split}
% \end{equation}
% and let $f \colon \mathbb{Z} \to \mathbb{R}_{\geq 0}$ be defined by
% \begin{equation}
% \label{eq:fdef}
% f(n) = 
% \begin{cases}
% (\log N)^k \frac{\varphi(W)^k}{W^k} &\text{if $n \in X$;} \\
% 0 & \text{otherwise.}
% \end{cases}
% \end{equation}
% Then
% \[
% \int_0^1 |\widehat{f}(\theta)|^{5/2} d\theta \leq K N^{3/2}.
% \]
% \end{proposition}
% \begin{proof}
% Taking $R = N^{1/1000}$ in~\cite[Proposition 4.2]{GreenTaoRestr} and $a_n = f(n)/\beta_R(n)$, where $\beta_R(n)$ is the enveloping sieve given by~\cite[Proposition 3.1]{GreenTaoRestr} we obtain that
% \[
% \int_0^1 |\widehat{f}(\theta)|^{5/2} d\theta \ll_m N^{3/2} \left(\frac{1}{N}\sum_{n \leq N} \frac{|f(n)|^2}{\beta_R(n)}\right)^{1/2} \ll N^{3/2} \left((\log N)^k \frac{\varphi(W)^k}{W^k} \frac{|X|}{N}\right)^{1/2} 
% \]
% by~\cite[Proposition 3.1(i)]{GreenTaoRestr}. Hence the claim follows from a standard upper sieve bound for $|X|$.
% \end{proof}

\begin{proof}[Proof of Theorem~\ref{th:GoldbachMaynard}]
Let $k = k(m)$, $\delta_0 = \delta_0(m)$, and $\rho = \rho(m)$ be as in Theorem~\ref{th:MaynPrimesinBohrSets}, and let $K = K(k, \rho/2)$, where $K(k, \rho/2)$ is as in Proposition~\ref{prop:GreenTaoLpnorm}. Let $\varepsilon > 0$ be small enough depending on $m$, let $w$ be large enough depending on $\varepsilon$ and $m$, and let $W = \prod_{p \leq w} p$. 

Let $N'$ be an odd positive integer, sufficiently large in terms of all the preceding quantities $m,k,\delta_0,\rho,K,\varepsilon,W$. Our goal is to find a representation
\[ N' = p_1 + p_2 + p_3, \]
where, for $j = 1, 2, 3$, $p_j$ are primes such that the interval $[p_j, p_j + H]$ contains at least $m$ primes. For $j = 1, 2, 3$, let $b_j$ be integers such that $1 \leq b_j \leq W$, $(b_j, W) = 1$, and $N' \equiv b_1 + b_2 + b_3 \pmod{W}$. Let
\[ N = \frac{N' - b_1 - b_ 2 - b_3}{W}. \]
Let $h_1^{(1)}, \dotsc, h_{k-1}^{(1)} \ll_m 1$ be as in Theorem~\ref{th:MaynPrimesinBohrSets} with $\chi = 1$. We can assume that $w$ is so large that $|h_i^{(1)}| < w/2$ for each $i$.

With these choices $w, b_1, h_j^{(1)}$ we define 
\begin{equation}
\label{eq:Xdef}
\begin{split}
X_1 = \{n \leq N \colon Wn+b_1 \in \mathbb{P}, & \quad Wn+b_1+Wh_i^{(1)} \in \mathbb{P} \text{ for } i = 1, \dotsc, m-1, \\
&\text{ and } \quad p \mid \prod_{i = m}^{k-1} (Wn+b_1+Wh_i^{(1)}) \implies p \geq N^{\rho/2}\},
\end{split}
\end{equation}
and let $f_1 \colon \mathbb{Z} \to \mathbb{R}_{\geq 0}$ be defined by
\begin{equation}
\label{eq:fdef}
f_1(n) = 
\begin{cases}
(\log N)^k \frac{\varphi(W)^k}{W^k} &\text{if $n \in X_1 \cap [0.2N, 0.4N)$;} \\
0 & \text{otherwise.}
\end{cases}
\end{equation}
Theorem~\ref{th:MaynPrimesinBohrSets} implies
\[ 
\sum_{\substack{0.2N \leq n < 0.4N}} f_1(n) \geq \frac{\delta_0}{10} N
\]
whereas Proposition~\ref{prop:GreenTaoLpnorm} applied with the linear forms
\[ \CL = \{Wn + b_1, Wn + b_1 + Wh_1^{(1)}, \cdots, Wn + b_1 + Wh_{k-1}^{(1)}\}  \]
implies
\[
\sum_{\xi \in G} |\widehat{f_1}(\xi)|^{5/2} \leq K
\]
since $|h_j^{(1)}| \leq |w|/2$, so that $\mathfrak{S}(\CL) \leq (W/\varphi(W))^k$.

Let further $\chi = \chi_{\Omega, \eta}$ be the Bohr cutoff associated to $f_1$ with $\delta = \delta_0/40$ from Theorem~\ref{thm:transfer}, with $|\Omega| \ll_m 1$, $1 \in \Omega$, and $1 \ll_m \eta < 0.05$. For $j =2, 3$, let $h_1^{(j)}, \dotsc, h_{k-1}^{(j)} \ll_m 1$ be as in Theorem~\ref{th:MaynPrimesinBohrSets} with $b = b_j$ and this choice of $\chi$. We can assume that $w$ is so large that $|h_i^{(j)}| < w/2$. With these choices $w, b_j, h_i^{(j)}$ we define, for $j = 2, 3$, $X_j$ and $f_j$ analogously to~\eqref{eq:Xdef} and~\eqref{eq:fdef}, but with $f_j$ now supported on $[N/4, N/2)$. For $t \in [N/4, N/2)$, Theorem~\ref{th:MaynPrimesinBohrSets} implies
\[
\begin{split}
\sum_{N/4 \leq n < N/2} f_j(n) \chi(t-n) &\geq \frac{\delta_0}{10} \left(\sum_{N/4 \leq n < N/2} \chi(t-n) + O\left(\frac{N}{w^{1/3}} + \varepsilon N \right)\right) \\
&\geq \frac{\delta_0}{30} \left(\sum_{n \in G} \chi(n) + O\left(\frac{N}{w^{1/3}} + \varepsilon N \right)\right),
\end{split}
\]
where the second inequality follows since $\chi$ is symmetric around $0$ and is essentially supported on $|n| \leq 0.1 N$, in the sense that
\begin{equation}
\label{eq:chi(n)essentialSupport}
\sum_{0.1N < n < 0.9N} \chi(n) \leq \eta \sum_{n \in G} \chi(n) 
\end{equation}
by Lemma~\ref{le:BohrCutOffProp}. When $w$ is large enough and $\varepsilon$ is small enough in terms of $m, \eta$ and $|\Omega|$ (the size of which depend only on $m$), this together with Lemma~\ref{le:BohrCutOffProp} implies that
\[
f_j \ast \chi(t) \geq \frac{\delta_0}{40}\Vert \chi \Vert_1.
\]
Furthermore Proposition~\ref{prop:GreenTaoLpnorm} implies that, for $j = 2, 3$,
\[
\sum_{\xi \in G} |\widehat{f_j}(\xi)|^{5/2} \leq K.
\]
Hence all the assumptions of Theorem~\ref{thm:transfer} are satisfied, and thus $f_1 * f_2 * f_3(N) \gg \delta^3$. In particular, there exists $n_1, n_2, n_3$ lying in the support of $f_1, f_2, f_3$, respectively, such that $n_1 + n_2 + n_3 \equiv 0 \pmod{N}$. By the definitions of $f_1,f_2,f_3$, we necessarily have $n_1 + n_2 + n_3 = N$, and moreover for $i=1,2,3$, $Wn_i + b_i$ are primes and so are $Wn_i + b_i + Wh_j^{(i)}$ for $1 \leq j \leq m-1$. This gives the desired representation
\[ N' = (Wn_1 + b_1) + (Wn_2 + b_2) + (Wn_3 + b_3), \]
once $H$ is large enough in terms of $m$.
\end{proof}

\begin{proof}[Proof of Theorem~\ref{th:GoldbachChen}]
Let $K = K(2, 1/2000)$, where $K(k, \rho)$ is as in Proposition~\ref{prop:GreenTaoLpnorm}. Let $w$ be a large parameter, and let $W = \prod_{p \leq w} p$. 

Let $N'\equiv 3\pmod{6}$ be a positive integer, sufficiently large in terms of $K,W$. Our goal is to find a representation
\[ N' = p_1 + p_2 + p_3, \]
where, for $j = 1, 2, 3$, $p_j+2$ has at most two prime factors. For $j = 1,2, 3$, let $b_j$ be integers such that $1 \leq b_j \leq W$, $(b_j, W) = (b_j+2,W) = 1$, and $N' \equiv b_1 + b_2 + b_3 \pmod{W}$. Let
\[ N = \frac{N' - b_1 - b_ 2 - b_3}{W}. \]
For $j =1, 2, 3$, we define
\[
\begin{split}
X_j = \{n \leq N \colon Wn+b_j \in \mathbb{P}, \quad Wn + b_j + 2 = P_2, \quad p \mid Wn+b_j+2 \implies p \geq N^{1/1000}\},
\end{split}
\]
and let $f_1 \colon \mathbb{Z} \to \mathbb{R}_{\geq 0}$ be defined by
\[
f_1(n) = 
\begin{cases}
(\log N)^2 \frac{\varphi(W)^2}{W^2} &\text{if $n \in X_1 \cap [0.2N, 0.4N)$;} \\
0 & \text{otherwise.}
\end{cases}
\]

Now Theorem~\ref{th:ChenPrimesinBohrSets} with $\chi = 1$ implies that
\[ 
\sum_{\substack{0.2N \leq n < 0.4 N}} f_1(n) \geq \frac{\delta_1}{10} N
\]
Let further $\chi = \chi_{\Omega, \eta}$ be the Bohr cutoff associated to $f_1$ with $\delta = \delta_1/40$ from Theorem~\ref{thm:transfer}, with $|\Omega| \ll 1$, $1\in \Omega$, and $1 \ll \eta < 0.05$. We define $f_j$ for $j = 2, 3$ as $f_1$ but with support $[N/4, N/2)$. Now Theorem~\ref{th:ChenPrimesinBohrSets} implies that, for $j = 2, 3$, and $t \in [N/4, N/2)$,
\[
\sum_{\substack{N/4 \leq n < N/2}} f_j(n) \chi(t-n) \geq \frac{\delta_1}{10} \left(\sum_{N/4 \leq n < N/2} \chi(t-n) + O\left(\frac{N}{w^{1/3}}\right)\right) \geq \frac{\delta_1}{30} \left(\sum_{n \in G} \chi(n) + O\left(\frac{N}{w^{1/3}}\right)\right)
\]
since $\chi(n)$ is essentially supported on $|n| \leq 0.1 N$ (see~\eqref{eq:chi(n)essentialSupport}) and is symmetric around $0$. When $w$ is large enough in terms of $\eta$ and $\Omega$ (sizes of which depend only on $m$), this and Lemma~\ref{le:BohrCutOffProp} imply that
\[
f_j \ast \chi(t) \geq \frac{\delta_1}{40}\Vert \chi \Vert_1.
\]
Furthermore Proposition~\ref{prop:GreenTaoLpnorm} implies that, for $j = 1, 2, 3$,
\[
\sum_{\xi \in G} |\widehat{f_j}(\xi)|^{5/2} \leq K.
\]

Hence all the assumptions of Theorem~\ref{thm:transfer} are satisfied, and thus $f_1 * f_2 * f_3(N) \gg \delta^3$. In particular, there exists $n_1, n_2, n_3$ lying in the support of $f_1, f_2, f_3$, respectively, such that $n_1 + n_2 + n_3 \equiv 0 \pmod{N}$. By the definitions of $f_1,f_2,f_3$, we necessarily have $n_1 + n_2 + n_3 = N$, and moreover for each $i=1,2,3$, $Wn_i + b_i$ is a prime and $Wn_i + b_i + 2$ has at most two prime factors. This gives the desired representation
\[ N' = (Wn_1 + b_1) + (Wn_2 + b_2) + (Wn_3 + b_3). \]
\end{proof}

\section{Weighted versions of Maynard's theorem and Chen's theorem}
\label{se:MaynardChenTheorems}
As discussed in the introduction, the celebrated result of Maynard~\cite{MaynardI} (obtained independently by Tao in an unpublished work) tells that, for each $m \geq 1$, there exists a constant $H = H(m)$ such that there exists infinitely many primes $p$ for which the interval $[p, p+H]$ contains at least $m$ primes. In a subsequent paper~\cite{MaynardII}, Maynard generalised the result to show that any subset of the primes which is well-distributed in arithmetic progressions (in a certain precise sense) contains many primes with bounded gaps, and also made an extension to linear forms representing primes.

In this section we state a slight variant of the main result of~\cite{MaynardII} in the case when the underlying set is weighted with weights $\omega_n \geq 0$. We also carefully state the dependencies between different parameters.

For a linear function $L(n) = l_1 n + l_2$, we define $\varphi_L(q) = \varphi(|l_1|q)/\varphi(|l_1|)$. Let us first state the needed hypotheses which correspond to~\cite[Hypothesis 1]{MaynardII}.
\begin{hypothesis}
\label{hyp:MaynardHyp}
For a sequence $(\omega_n)$, a set of $k$ admissible linear forms $\mathcal{L}$, and real numbers $x \geq 2, \theta \in (0, 1)$ and $C_H > 0$, we formulate the following hypothesis. 
\begin{enumerate}[(1)]
\item $(\omega_n)$ is well-distributed in arithmetic progressions: We have
\[
\sum_{r \leq x^\theta} \max_c \Biggl|\sum_{\substack{x \leq n < 2x \\ n \equiv c \pmod{r}}} \omega_n - \frac{1}{r} \sum_{\substack{x \leq n < 2x}} \omega_n\Biggr| \leq C_H \frac{\sum_{\substack{x \leq n < 2x}} \omega_n}{(\log x)^{101k^2}}.
\]
\item Primes represented by linear forms in $\mathcal{L}$ are well-distributed in arithmetic progressions: For any $L \in \mathcal{L}$, we have
\[
\sum_{r \leq x^\theta} \max_{(L(c), r) = 1} \Biggl|\sum_{\substack{x \leq n < 2x \\ n \equiv c \pmod{r} \\ L(n) \in \mathbb{P}}} \omega_n - \frac{1}{\varphi_L(r)} \sum_{\substack{x \leq n < 2x \\ L(n) \in \mathbb{P}}} \omega_n \Biggr| \leq C_H \frac{\sum_{\substack{x \leq n < 2x}} \omega_n}{(\log x)^{101k^2}}.
\]
\item $(\omega_n)$ is not too concentrated in any arithmetic progression: For any $r \leq x^\theta$ and any $c$, we have
\[
\sum_{\substack{x \leq n < 2x \\ n \equiv c \pmod{r}}} \omega_n \leq C_H \frac{1}{r} \sum_{\substack{x \leq n < 2x}} \omega_n.
\]
\end{enumerate}
\end{hypothesis}

The slight variant of Maynard's main theorem~\cite[Theorem 3.1]{MaynardII} now states
\begin{theorem}
\label{th:Maynard}
Let $\alpha > 0, \theta \in (0, 1)$ and $C_H > 0$. There exist a constant $C = C(\alpha, \theta)$ such that, for any $k \geq C$ there exist positive constants $x_0 = x_0(\alpha, \theta, k, C_H), \delta_0 = \delta_0(\alpha, \theta, k)$ and $\rho = \rho(\alpha, \theta, k)$ such that the following holds.

Let $(\omega_n)$ be a sequence of non-negative real numbers, let $\mathcal{L} = \{L_1, \dotsc, L_k\}$ be an admissible set of $k$ linear functions, and let $x \geq x_0$ be an integer. Assume that the coefficients of $L_i(n) = a_i n + b_i$ satisfy $1 \leq a_i, b_i \leq x^\alpha$ for all $1 \leq i \leq k$, and assume that $k \leq (\log x)^\alpha$.

If Hypothesis~\ref{hyp:MaynardHyp} holds and $\delta > 1/(\log k)$ is such that
\begin{equation}
\label{eq:MaynardHypExtra}
\frac{1}{k} \sum_{L \in \mathcal{L}} \frac{\varphi(a_i)}{a_i} \sum_{\substack{x \leq n < 2x \\ L(n) \in \mathbb{P}}} \omega_n \geq \frac{\delta}{\log x} \sum_{x \leq n < 2x} \omega_n,
\end{equation}
then
\[
\sum_{\substack{x \leq n < 2x \\ \#(\{L_1(n), \dotsc, L_k(n)\} \cap \mathbb{P}) \geq C^{-1} \delta \log k \\ p \mid L_1(n) \dotsm L_k(n) \implies p > x^{\rho}}} \omega_n \geq \delta_0 \frac{\mathfrak{S}(\mathcal{L})}{(\log x)^k \exp(Ck)} \sum_{x \leq n < 2x} \omega_n,
\]
where $\FS(\CL)$ is defined as in \eqref{eq:singseries}.
%\begin{equation}
%\label{eq:singseries}
%\mathfrak{S}(\mathcal{L}) = \prod_p \left(1-\frac{\#\{1 \leq n \leq p \colon p \mid \prod_{i=1}^k L_i(n)\}}{p}\right)\left(1-\frac{1}{p}\right)^{-k} %\gg \frac{1}{\exp(O(k))}.
%\end{equation}
\end{theorem}

\begin{proof}
The proof is the same as Maynard's~\cite[Proof of Theorem 3.1]{MaynardII}. Introducing the weights $\omega_n$ makes no difference once one replaces $\# \mathcal{A}(x)$ in~\cite{MaynardII} by the weighted version $\sum_{x \leq n < 2x} \omega_n$ etc. Furthermore, to see that the constants $\delta_0$ and $\rho$ do not depend on $C_H$, notice that Hypothesis~\ref{hyp:MaynardHyp}(1, 2) imply~\cite[Hypothesis 1(1,2)]{MaynardII} with implied constant one once $x$ is large enough in terms of $C_H$. On the other hand, in~\cite[Proof of Theorem 3.1]{MaynardII},~\cite[Hypothesis 1(3)]{MaynardII} is only used together with~\cite[Hypothesis 1(1) or (2)]{MaynardII} to dispose of some divisor functions through the Cauchy-Schwarz inequality (see~\cite[Formulas (9.2)--(9.3)]{MaynardII} for a typical example). In these situations one also wins a power of $\log x$ and thus can take the implied constant in the resulting bounds to be one once $x$ is large enough in terms of $C_H$. Hence none of the implied constants in the proof of Maynard's theorem depend on $C_H$ once $x$ is large enough in terms of $C_H$.
\end{proof}

Next we formulate a similar general version of Chen's theorem. We will need the notion of a well-factorable function of level $R$ by which we mean a function $\lambda: \mathbb{N} \cap [1, R] \to [-1,1]$ such that, for any $S, T \geq 1$ with $ST = R$, we can write $\lambda = \gamma \ast \delta$ with $1$-bounded functions $\gamma$ and $\delta$ supported respectively on $[1, S]$ and $[1, T]$. 
\begin{hypothesis}\label{hyp:ChenHyp}
For $\varepsilon \in (0,0.1)$, a sequence $(\omega_n)$ of non-negative real numbers, a set of two admissible linear forms $\CL = \{L_1, L_2\}$ with $L_i(n) = u_in + v_i$, and a real number $x \geq 2$, we formulate the following hypotheses.
\begin{enumerate}
\item
Primes represented by $L_1$ are well-distributed in arithmetic progressions: We have
\[ 
\sum_{\substack{r \\ (r,u_2(u_2v_1 - u_1v_2))=1}} \mu(r)^2\lambda_r \Biggl( \sum_{\substack{x \leq n < 2x \\ r \vert L_2(n) \\ L_1(n) \in \Prime}}  \omega_n -  \frac{u_1}{\varphi(r u_1)} \sum_{\substack{x \leq n < 2x}} \frac{\omega_n}{\log L_1(n)} \Biggr) \leq \frac{\sum_{x \leq n < 2x} \omega_n}{(\log x)^{10}}
\]
whenever $\lambda$ is a well-factorable function of level $x^{1/2-\varepsilon}$ or $\lambda = 1_{p \in [P, P')} \ast \lambda'$, where $\lambda'$ is a well-factorable function of level $x^{1/2-\varepsilon}/P$ and $2P \geq P' \geq P \in [x^{1/10}, x^{1/3-\varepsilon}]$.
\item
Almost primes represented by $L_2$ are well-distributed in arithmetic progressions: We have, for $j = 1, 2$,
\[ 
\begin{split}
&\sum_{\substack{r \\ (r,u_1(u_1v_2 - u_2v_1))=1}} \mu(r)^2 \lambda_r \Biggl( \sum_{\substack{x \leq n < 2x \\ r \vert L_1(n) \\ L_2(n) \in B_j}}  \omega_n - \frac{1}{\varphi_{L_2}(r)} \sum_{\substack{x \leq n < 2x \\ L_2(n) \in B_j}} \omega_n \Biggr) \leq \frac{\sum_{x \leq n < 2x} \omega_n}{(\log x)^{10}}
\end{split}
\]
whenever $\lambda$ is a well-factorable function of level $x^{1/2-\varepsilon}$, where
\begin{equation}\label{eq:ChenB} 
\begin{split}
B_1 &= \{n = p_1p_2p_3 \ \vert \ x^{1/10} \leq p_1 < x^{1/3-\varepsilon}, x^{1/3-\varepsilon} \leq p_2 \leq (L_2(2x)/p_1)^{1/2}, p_3 \geq x^{1/10} \} \\
\text{and} \quad B_2 &= \{n = p_1p_2p_3 \ \vert \ x^{1/3-\varepsilon} \leq p_1 \leq p_2 \leq (L_2(2x)/p_1)^{1/2}, p_3 \geq x^{1/10}\}.
\end{split}
\end{equation}
\item
$(\omega_n)$ is not concentrated in $B_j$: We have, for $j=1,2$,
\[
\sum_{\substack{x \leq n < 2x \\ L_2(n) \in B_j}} \omega_n  \leq (1+o(1)) \frac{|B_j \cap [L_2(x), L_2(2x))|}{\varphi(u_2)} \cdot \frac{1}{x} \sum_{x \leq n < 2x} \omega_n, 
\]
\end{enumerate}
\end{hypothesis}

Note that the factor $u_1/(\varphi(ru_1) \log L_1(n))$ in the first hypothesis is the probability that a randomly chosen $n \in [x, 2x)$ satisfies $r \mid L_2(n)$ and $L_1(n) \in \Prime$. Note also that it is straightforward to find the density of $B_j$: If $u_2, v_2 \leq x^{o(1)}$ then
\[ |B_j \cap [L_2(x), L_2(2x))| = (\delta(B_j)+o(1))\frac{u_2x}{\log x}, \]
where
\begin{equation}\label{eq:density-Bj}
\delta(B_1) = \int_{1/10}^{1/3-\varepsilon} \int_{1/3-\varepsilon}^{(1-\alpha_1)/2} \frac{d\alpha_2 d\alpha_1}{\alpha_1\alpha_2(1-\alpha_1-\alpha_2)}, \quad \delta(B_2) = \int_{1/3-\varepsilon}^{1/3} \int_{\alpha_1}^{(1-\alpha_1)/2} \frac{d\alpha_2 d\alpha_1}{\alpha_1\alpha_2(1-\alpha_1-\alpha_2)}.
\end{equation}

To see that the coprimality conditions $(r,u_2(u_2v_1 - u_1v_2)) = 1$ and $(r, u_1(u_1v_2 - u_2v_1)) = 1$ occur naturally, note that if $r$ and $u_i$ share a common prime divisor $p$, then $p\nmid L_i(n)$ for all $n$ by the admissibility of $L_i$, and thus the sum over those $n$ satisfying $r \vert L_i(n)$ is empty. Similarly, if $(r,u_i)=1$ but $r$ and $u_2v_1 - u_1v_2$ share a common prime divisor $p < x^{1/10}$, then $p \vert L_i(n)$ implies $p \vert L_j(n)$ (where $j = 3 - i$), and thus the sum over those $n$ satisfying $r \vert L_i(n)$ and $L_j(n) \in B$ (or $L_j(n) \in \Prime$) is empty.

\begin{theorem}\label{thm:chen}
There exist positive constants $\delta_0, \varepsilon$ and $x_0$ such that the following holds. Let $(\omega_n)$ be a sequence of non-negative real numbers, $\CL = \{L_1,L_2\}$ be an admissible set of two linear functions, and let $x \geq x_0$. Assume that the coefficients of $L_i(n) = u_in + v_i$ satisfy $1 \leq u_i, v_i \leq x^{o(1)}$, and that Hypothesis~\ref{hyp:ChenHyp} holds. Then
\[ \sum_{\substack{x \leq n < 2x \\ L_1(n) \in \Prime \\ L_2(n) = P_2 \\ p \mid L_2(n) \implies p \geq x^{1/10} }} \omega_n \geq \delta_0 \frac{ \FS(\CL) }{ (\log x)^2} \sum_{x \leq n < 2x} \omega_n - O\left(x^{0.9} \max_n \omega_n\right) \]
where $\FS(\CL)$ is as in~\eqref{eq:singseries}.
\end{theorem}
Since the proof is essentially Chen's sieving device written in general terms, we postpone its proof to Appendix~\ref{app:Chen}.

\section{Technical reductions}
\label{se:technRed}
The conclusion of Maynard's theorem does not quite correspond to the conclusion we want in Theorem~\ref{th:MaynPrimesinBohrSets}. However, we can quickly deduce Theorem~\ref{th:MaynPrimesinBohrSets} from the following variant which is more apt for an application of Maynard's theorem.
\begin{proposition}
\label{prop:AlmMaynPrimesinBohrSets}
For any positive integer $m$, there exist a positive integer $k = k(m)$ and positive constants $\delta_1 = \delta_1(m)$ and $\rho = \rho(m)$ such that the following holds. Let $\chi: \Z \rightarrow \R_{\geq 0}$ be a function with Fourier complexity at most $M$ for some $M \geq 1$, let $W = \prod_{p \leq w} p$ and let $(b, W) = 1$. There exists a positive constant $N_0 = N_0(m, M, w)$ such that, for any distinct integers $h_1, \dotsc, h_{k}$ with $|h_j| < w/2$, any $N \geq N_0$ and $|t| \leq 5 N$,
\[ 
\sum_{\substack{ N \leq n < 2N \\ |\{W(n+h_i)+b\} \cap \mathbb{P}| \geq m \\ p \mid \prod_{i = 1}^{k} (W(n+h_i)+b) \implies p \geq N^{\rho}}} \chi(t-n) \geq \delta_1 \frac{1}{(\log N)^k} \frac{W^k}{\varphi(W)^k} \left(\sum_{N \leq n < 2N} \chi(t-n) + O\left(\frac{M^2 N}{w^{1/2}}\right)\right)
\]
\end{proposition}

\begin{proof}[Proof that Proposition~\ref{prop:AlmMaynPrimesinBohrSets} implies Theorem~\ref{th:MaynPrimesinBohrSets}]
Let $k = k(m)$, $\delta_1 = \delta_1(m)$, and $\rho = \rho(m)$  be as in Proposition~\ref{prop:AlmMaynPrimesinBohrSets}. Let $\alpha_1, \dotsc, \alpha_M$ be the phases appearing in the Fourier expansion of $\chi$. By the simultaneous version of the Dirichlet approximation theorem, we can find $k$ distinct positive integers $h'_j \ll_{M, \varepsilon, m} 1$ such that
\[
\Vert \alpha_i h'_j \Vert \leq \frac{\varepsilon}{M^2} \quad \text{for every $i = 1, \dotsc, M$ and $j = 1, \dotsc, k$.}
\] 
These choices ensure that, whenever $n - n' \in \{h_1', \cdots, h_k'\}$ we have
\begin{equation}\label{eq:chi(n)-chi(n')} 
| \chi(n) - \chi(n')| \ll \varepsilon.
\end{equation}
We can assume that $w$ is so large in terms of $M, \varepsilon$ and $m$ that $|h'_j| < w/2$ for all $j$ and $w^{1/6}$ is at least $2M^2$ times the implied constant in the conclusion of Proposition~\ref{prop:AlmMaynPrimesinBohrSets}. By Proposition~\ref{prop:AlmMaynPrimesinBohrSets} we see that, for any $|t| \leq 5N$, 
\[ 
\sum_{\substack{ N \leq n < 2N \\ |\{W(n+h'_i)+b\} \cap \mathbb{P}| \geq m \\ p \mid \prod_{i = 1}^{k} (W(n+h'_i)+b) \implies p \geq N^{\rho}}} \chi(t-n) \geq \delta_1 \frac{1}{(\log N)^k} \frac{W^k}{\varphi(W)^k} \left(\sum_{N \leq n < 2N} \chi(t-n) - \frac{N}{2w^{1/3}}\right).
\]
%We can assume that the main term is at least twice the error term since otherwise Theorem~\ref{th:MaynPrimesinBohrSets} is trivial.
We get that, for some $\mathcal{J} \subseteq \{1, \dotsc, k\}$ with $\#\mathcal{J} = m$, 
\[ 
\sum_{\substack{ N \leq n < 2N \\ W(n+h'_j)+b \in \mathbb{P} \text{ for each $j \in \mathcal{J}$} \\ p \mid \prod_{i = 1}^{k} (W(n+h'_i)+b) \implies p \geq N^{\rho}}} \chi(t-n) \geq \frac{\delta_1}{{k \choose m}} \cdot \frac{1}{(\log N)^k} \frac{W^k}{\varphi(W)^k} \left(\sum_{N \leq n < 2N} \chi(t-n) -\frac{N}{2w^{1/3}}\right).
\]
Let $r \in \mathcal{J}$ be such that $h'_r$ is the minimal among $h'_j$ with $j \in \mathcal{J}$. We take $h_1, \dots, h_{k-1}$ to be any choice (unique up to permutation) such that
\[
\begin{split}
\{h_i \colon i = 1, \dotsc, m-1\} &= \{h_i'-h'_r \colon i \in \mathcal{J}\setminus \{r\}\} \\ \text{and} \quad \{h_i \colon i = m, \dotsc, k-1\} &= \{h_i'-h'_r \colon i \in \{1, \dotsc, k\} \setminus \mathcal{J}\}.
\end{split}
\]
Substituting $n' = n+h_r'$, we see that
\begin{equation}
\label{eq:sumchi(t-n+h_k')}
\sum_{\substack{ N+h_r' \leq n' < 2N + h_r' \\ Wn'+b \in \mathbb{P} \\ Wn' + Wh_i + b \in \mathbb{P} \text{ for $i = 1, \dotsc, m-1$} \\ p \mid \prod_{i = m}^{k-1} (Wn' + Wh_i+ b) \implies p \geq N^{\rho}}} \chi(t-n'+h_r') \geq \frac{\delta_1}{{k \choose m}} \cdot  \frac{1}{(\log N)^k} \frac{W^k}{\varphi(W)^k} \left(\sum_{N \leq n < 2N} \chi(t-n) - \frac{N}{2w^{1/3}}\right).
\end{equation}
By \eqref{eq:chi(n)-chi(n')} we may replace the summand $\chi(t-n' + h_r')$ above by $\chi(t-n')$ using a standard sieve bound for the number of elements counted on the left hand side of~\eqref{eq:sumchi(t-n+h_k')}, getting that 
\[
\sum_{\substack{ N+h_r' \leq n' < 2N + h_r' \\ Wn'+b \in \mathbb{P} \\ Wn' + Wh_i + b \in \mathbb{P} \text{ for $i = 1, \dotsc, m-1$} \\ p \mid \prod_{i = m}^{k-1} (Wn' + Wh_i+ b) \implies p \geq N^{\rho}}} \chi(t - n') \geq \frac{\delta_1}{{k \choose m}} \cdot \frac{1}{(\log N)^k} \frac{W^k}{\varphi(W)^k} \left(\sum_{N \leq n < 2N} \chi(t-n) - \frac{N}{2w^{1/3}} + O(\varepsilon N)\right)
\]
with the implied constant depending only on $k$ and $\rho$ and thus only on $m$. Theorem~\ref{th:MaynPrimesinBohrSets} follows with $\delta_0 = \delta_1/(2{k \choose m})$ through noting that the terms with $n' \in [2N, 2N + h_r')$ on the left hand side contribute at most $M^2 h_r'$.
\end{proof}

Since Bohr sets (and in general functions with bounded Fourier complexity) are not equidistributed in arithmetic progressions, we cannot apply Maynard's theorem to the situation in Proposition~\ref{prop:AlmMaynPrimesinBohrSets} directly, but we need to be careful with our choice of the sequence $\omega_n$ to which we apply Maynard's theorem. In particular the moduli $q_i \mid Q$ in the Fourier expansion of $\chi$ in Lemma~\ref{le:chi-Fourier} are problematic, and for this reason we will split into residue classes $\pmod{Q}$.

%For $Q$ from Lemma~\ref{le:chi-Fourier}, $W = \prod_{p \leq w} p$, $(b, W) = 1$, and $|h_j| < w/2$ for each $j$, write
%\[
%\mathcal{C}_M = \{c_0 \pmod{Q} \colon (Wc_0 + Wh_i +b, Q) = 1 \text{ for every $i = 1, \dotsc, k$}\}.
%\] 
%Notice that, since $|h_i| < w/2$, by the Chinese reminder theorem
%\begin{equation}
%\label{eq:|C|est}
%|\mathcal{C}_M| = \prod_{\substack{p \mid Q \\ p \leq w}} p \cdot \prod_{p \mid Q, p > w} (p-k) = Q \prod_{p \mid Q, p > w} \left(1-\frac{k}{p}\right).
%\end{equation}

In Section~\ref{sec:MaynApplication} we shall use Maynard's theorem (Theorem~\ref{th:Maynard}) and exponential sum estimates (which we will state in Section~\ref{sec:expsumest}) to prove the following proposition.
\begin{proposition}
\label{prop:AlmMaynPrimesinBohrSetsModQ}
For any positive integer $m$, there exist a positive integer $k = k(m)$ and positive constants $\delta_1 = \delta_1(m)$, $\rho = \rho(m)$ and $A = A(m)$ such that the following holds. Let $\chi: \Z \rightarrow \R_{\geq 0}$ be a function with Fourier complexity at most $M$ for some $M \geq 1$. Let $W = \prod_{p \leq w} p$, let $(b, W) = 1$, and let $N \geq N_0(m, M, w)$ be large. Let $Q$ be from Lemma~\ref{le:chi-Fourier} corresponding to $A$.  Then, for any distinct integers $h_1, \dotsc, h_{k}$ with $|h_j| < w/2$, any $|t| \leq 5N$ and $c_0 \in \mathcal{C}_M$,
\[ 
\sum_{\substack{ N \leq n < 2N \\ n \equiv c_0 \pmod{Q} \\ |\{W(n+h_i)+b\} \cap \mathbb{P}| \geq m \\ p \mid \prod_{i = 1}^{k} (W(n+h_i)+b) \implies p \geq N^{\rho}}} \chi(t-n) \geq \delta_1 \frac{1}{(\log N)^k} \frac{W^k}{\varphi(W)^k} \frac{Q}{|\mathcal{C}_M|} \left(\sum_{\substack{N \leq n < 2N \\ n \equiv c_0 \pmod{Q}}} \chi(t-n) + O\left(\frac{N}{Qw^{10}}\right)\right),
\]
where
\[
\mathcal{C}_M = \{c_0 \pmod{Q} \colon (Wc_0 + Wh_i +b, Q) = 1 \text{ for every $i = 1, \dotsc, k$}\}.
\] 
\end{proposition}

Notice that, since $|h_i| < w/2$, by the Chinese reminder theorem
\begin{equation}
\label{eq:|C|est}
|\mathcal{C}_M| = Q \prod_{p \mid Q, p > w} \left(1-\frac{k}{p}\right).
\end{equation}

Let us next state a similar proposition that we shall prove using Chen's theorem (Theorem~\ref{thm:chen}). 

%For $Q$ from Lemma~\ref{le:chi-Fourier}, $W = \prod_{p \leq w} p$, $(b, W) = (b+2, W) = 1$, write
%\[ \CC_C = \{c_0 \pmod Q, (W c_0 + b, Q) = (W c_0 + b + 2, Q) = 1\}. \]
%By the Chinese reminder theorem
%\[
%|\CC_C| = Q \prod_{\substack{p \mid Q \\ p > w}} \left(1-\frac{2}{p}\right).
%\]

\begin{proposition}
\label{prop:ChenPrimesinBohrSetsModQ}
Let $\chi: \Z \rightarrow \R_{\geq 0}$ be a function with Fourier complexity at most $M$ for some $M \geq 1$. Let $W = \prod_{p \leq w} p$, let $(b, W) = (b+2, W) = 1$, and let $N \geq N_0(M, w)$ be large. Let $Q$ be from Lemma~\ref{le:chi-Fourier} corresponding to some large enough $A$. Then for any $|t| \leq 5N$, and $c_0 \in \mathcal{C}_C$,
\[ 
\sum_{\substack{ N \leq n < 2N \\ n \equiv c_0 \pmod{Q} \\ Wn + b \in \Prime \\ Wn + b + 2 = P_2 \\ p \mid Wn + b + 2 \implies p \geq N^{1/100}}} \chi(t-n) \geq \delta_1 \frac{1}{(\log N)^2} \frac{W^2}{\varphi(W)^2} \frac{Q}{|\mathcal{C}_C|} \left( \sum_{\substack{N \leq n < 2N \\ n \equiv c_0 \pmod{Q}}} \chi(t-n) + O\left( \frac{N}{Q(\log N)^{100}} \right) \right),
\]
for some absolute constant $\delta_1 > 0$, where
\[
\CC_C = \{c_0 \pmod Q, (W c_0 + b, Q) = (W c_0 + b + 2, Q) = 1\}.
\]
\end{proposition}

%\begin{proof}[Proof that Proposition~\ref{prop:ChenPrimesinBohrSetsModQ} implies Proposition~\ref{prop:ChenPrimesinBohrSets}]
%In view of Lemma~\ref{le:sumoverc0}, it suffices to verify \eqref{eq:sumoverc_0inC} for $\CQ = \CC$. This follows from the same argument used in treating Maynard primes.
%\end{proof}

To show that Propositions \ref{prop:AlmMaynPrimesinBohrSetsModQ} and~\ref{prop:ChenPrimesinBohrSetsModQ} imply Proposition~\ref{prop:AlmMaynPrimesinBohrSets} and Theorem~\ref{th:ChenPrimesinBohrSets}, we use the following lemma allowing us to sum over all the residue classes in $\mathcal{C}_M$ and $\mathcal{C}_C$.

\begin{lemma}
\label{le:sumoverc0}
Let $\chi$ be a function of Fourier complexity at most $M$ for some $M \geq 1$, and let $N, Q$ be positive integers with $N \geq 2Q^2$. Let also $\mathcal{Q}$ be a collection of residue classes modulo $Q$ such that, for all $1 \neq q \mid Q$ and $(a,q) = 1$, one has
\begin{equation}
\label{eq:sumoverc_0inC}
\sum_{c_0 \in \mathcal{Q}} e\left(\frac{a }{q}c_0 \right) = O\left(\eta |\mathcal{Q}|\right),
\end{equation}
for some $\eta > 0$. Then
\begin{equation}
\label{eq:sumoverc_0claim}
\frac{Q}{|\mathcal{Q}|} \sum_{\substack{c_0 \in \mathcal{Q}}} \sum_{\substack{N \leq n < 2N \\ n \equiv c_0 \pmod{Q}}} \chi(t-n) \geq \sum_{\substack{N \leq n < 2N}} \chi(t-n) + O\left( \eta M^2 N +  Q M^2 N^{1/2} \right).
\end{equation}
\end{lemma}

\begin{proof}
By Definition~\ref{def:Fourier-complexity}, we have the Fourier expansion 
\[
\chi(t-n) = \sum_{i=1}^M b_i e\left(\alpha_i (t-n)\right),
\]
for some $|b_i| \leq M$ and $\alpha_i \in \R/\Z$. For each $1 \leq i \leq M$, we may find integers $0 \leq a_i < q_i \leq N^{1/2}$ with $(a_i,q_i) = 1$ such that $|\alpha_i - a_i/q_i| \leq 1/(q_iN^{1/2})$. 

Let us first consider the contribution of those $i$ with $q_i = 1$ to the left hand side of~\eqref{eq:sumoverc_0claim}. This contribution is, using Lemma~\ref{le:GeomSerI},
\[
\begin{split}
\Sigma_1 &:= \frac{Q}{|\mathcal{Q}|} \sum_{\substack{c_0 \in \mathcal{Q}}}  \sum_{\substack{1 \leq i \leq M \\ q_i = 1}} b_i \sum_{\substack{N \leq n < 2N \\ n \equiv c_0 \pmod{Q}}}  e(\alpha_i(t-n)) \\
& = \frac{Q}{|\mathcal{Q}|} \sum_{\substack{c_0 \in \mathcal{Q}}} \sum_{\substack{1 \leq i \leq M \\ q_i = 1}} b_i  \left(\frac{1}{Q} \sum_{\substack{N \leq n < 2N}}  e(\alpha_i(t-n)) + O(N^{1/2})\right) \\
& =\sum_{\substack{1 \leq i \leq M \\ q_i = 1}} b_i  \sum_{\substack{N \leq n < 2N}}  e\left(\alpha_i(t-n)\right) + O(Q M^2 N^{1/2}).
\end{split}
\]
By Lemma~\ref{le:GeomSerII} we can extend the sum to go over all $1 \leq i \leq M$, at the cost of an error of size $M^2 \max_i q_i \ll M^2 N^{1/2}$, getting
\[
\begin{split}
\Sigma_1  &= \sum_{\substack{N \leq n < 2N}} \sum_{1 \leq i \leq M} b_i  e\left( \alpha_i (t-n)\right) + O\left( Q M^2 N^{1/2} \right) \\
&= \sum_{\substack{N \leq n < 2N}} \chi(t-n) + O\left( QM^2 N^{1/2} \right).
\end{split}
\]

Hence we are finished if we can show that, for each $i$ such that with $q_i > 1$, we have
\begin{equation}
\label{eq:sumc_0claim}
\frac{Q}{|\mathcal{Q}|} \sum_{\substack{c_0 \in \mathcal{Q}}} \sum_{\substack{N \leq n < 2N \\ n \equiv c_0 \pmod{Q}}} e\left( \alpha_i(t-n) \right) = O\left( \eta N + Q N^{1/2} \right).
\end{equation}
In case $q_i \nmid Q$, we have $q_i/(q_i, Q) > 1$, and thus, by Lemma~\ref{le:GeomSerII}, the left hand side is $O(Q q_i) = O(Q N^{1/2})$.

In case $q_i \mid Q$, writing $\alpha_i = a_i/q_i + \beta_i$, the sum over $n$ on the left hand side of~\eqref{eq:sumc_0claim} equals
\[
\begin{split}
&e\left(\frac{a_i}{q_i}(t-c_0) \right) \sum_{\substack{N \leq n < 2N \\ n \equiv c_0 \pmod{Q}}} e\left(\beta_i(t-n) \right) \\
&=e\left(\frac{a_i}{q_i}(t-c_0) \right) \cdot \frac{1}{Q}\left( \sum_{\substack{N \leq n < 2N}} e\left(\beta_i(t-n) \right)\right) + O(N^{1/2}) 
\end{split}
\]
by Lemma~\ref{le:GeomSerI}. Hence the left hand side of~\eqref{eq:sumc_0claim} equals
\[
\begin{split}
\sum_{\substack{N \leq n < 2N}} e\left(\beta_i(t-n) + \frac{a_i}{q_i}t\right) \frac{1}{|\mathcal{Q}|} \sum_{\substack{c_0 \in \mathcal{Q}}} e\left(-\frac{a_i}{q_i} c_0 \right) + O(Q N^{1/2}), 
\end{split}
\]
and~\eqref{eq:sumc_0claim} follows from the assumption~\eqref{eq:sumoverc_0inC}.
\end{proof}

In order to show that~\eqref{eq:sumoverc_0inC} holds for $\mathcal{Q} = \mathcal{C}_M$  and for $\mathcal{Q} = \mathcal{C}_C$, we shall use the following elementary lemma related to a certain modification of Ramanujan sums.
\begin{lemma}
\label{le:RamanujanSum}
Let $q$ be a natural number, $(a, q) = 1$ and let $P(n)$ be a polynomial with integer coefficients. Write $\rho(n) = \#\{k \pmod{n} \colon P(k) \equiv 0 \pmod{n}\}$. Then
\[
\left|\sum_{\substack{n \pmod{q} \\ (P(n), q) = 1}} e\left(\frac{an}{q}\right)\right| \leq \rho(q).
\]
\end{lemma}
\begin{proof}
By M\"obius inversion,
\[
\sum_{\substack{n \pmod{q} \\ (P(n), q) = 1}} e\left(\frac{an}{q}\right) = \sum_{\substack{d \mid q}} \mu(d) \sum_{\substack{n \pmod{q} \\ P(n) \equiv 0 \pmod{d}}} e\left(\frac{an}{q}\right).
\]
For a fixed $d \mid q$, write $x_1, \dotsc, x_{\rho(d)}$ for the roots of $P(n) \pmod{d}$. Then
\begin{equation}
\label{eq:sumP(n)=0modd}
\sum_{\substack{n \pmod{q} \\ P(n) \equiv 0 \pmod{d}}} e\left(\frac{an}{q}\right) = \sum_{i = 1}^{\rho(d)} \sum_{\substack{n \pmod{q} \\ n \equiv x_i \pmod{d}}} e\left(\frac{an}{q}\right) = \sum_{i = 1}^{\rho(d)} e\left(\frac{ax_i}{q}\right) \sum_{\substack{k \pmod{q/d}}} e\left(\frac{ak}{q/d}\right),
\end{equation}
where we have written $n = x_i + k d$.
The last sum vanishes unless $d = q$ in which case~\eqref{eq:sumP(n)=0modd} has absolute value at most $\rho(q)$, and the claim follows.
\end{proof}

\begin{proof}[Proof that Proposition \ref{prop:AlmMaynPrimesinBohrSetsModQ} implies Proposition~\ref{prop:AlmMaynPrimesinBohrSets}]
We may assume that $w$ is large enough in terms of $m$, since otherwise the error term dominates and the claim is trivial. By Lemma~\ref{le:sumoverc0} it remains to show ~\eqref{eq:sumoverc_0inC} for $\mathcal{Q} = \mathcal{C}_M$ and $1 \neq q \mid Q$ with $\eta = w^{-1/2}$. Writing $R(n) = \prod_{i = 1}^k (Wn + Wh_i + b)$, \eqref{eq:sumoverc_0inC} reduces to
\begin{equation}
\label{eq:sumoverc_0}
\sum_{\substack{c_0 \pmod{Q} \\ (R(c_0), Q) = 1}} e\left(\frac{a }{q}c_0 \right) = O\left(\frac{Q}{w^{1/2}} \prod_{\substack{p \mid Q \\ p > w}} \left(1-\frac{k}{p}\right)\right).
\end{equation}
We can uniquely decompose $Q = q q' Q'$, where $(Q', q) = 1$ and $p \mid q' \implies p \mid q$. Then, when $c_1$ and $c_2$ run respectively through residue classes $\pmod{q'Q'}$ and $\pmod{q}$, $c_1 q + c_2 Q'$ runs through residue classes $\pmod{Q}$. Writing $c_0$ in this form, the left hand side of~\eqref{eq:sumoverc_0} becomes
\begin{equation}
\label{eq:sumoverc_0LHS}
\sum_{\substack{c_1 \pmod{q'Q'} \\ (R(c_1q), Q') = 1}} \sum_{\substack{c_2 \pmod{q} \\ (R(c_2 Q'), q) = 1}} e\left(\frac{aQ'}{q} c_2\right) 
\end{equation}

Since $R(n)$ is always co-prime to $W$, Lemma~\ref{le:RamanujanSum} implies that the inner sum in~\eqref{eq:sumoverc_0LHS} vanishes unless $(q, W) = 1$. Furthermore in this case it has absolute value at most 
\[ \#\{c_2 \pmod{q}: R(c_2Q') \equiv 0 \pmod{q} \} \leq  k^{\Omega(q)} \leq q^{1/3}, \]
since $p \mid q \implies p > w$ and $w$ is large enough. Hence we obtain that the absolute value of~\eqref{eq:sumoverc_0LHS} is at most
\[
\begin{split}
\sum_{\substack{c_1 \pmod{q'Q'} \\ (R(c_1q), Q') = 1}} q^{1/3} =  q' q^{1/3} \sum_{\substack{c_1 \pmod{Q'} \\ (R(c_1q), Q') = 1}} 1.
\end{split}
\]
By the definition of $R(n)$, $R(n)$ is always co-prime to $W = \prod_{p \leq w} p$, and for every $p >w$, $R(n) \equiv 0 \pmod{p}$ has $k$ incongruent solutions $\pmod{p}$ (since $|h_i| < w/2$ for every $i$). Hence the absolute value of~\eqref{eq:sumoverc_0LHS} is at most
\[
q^{1/3} q' Q' \prod_{p\mid Q', p>w} \left(1 - \frac{k}{p} \right) \leq \frac{Q}{q^{1/2}} \prod_{p \mid Q, p > w} \left(1-\frac{k}{p}\right),
\]
and~\eqref{eq:sumoverc_0} follows since $q > 1$ and $(q, W) = 1$, so that $q > w$.
\end{proof}

\begin{proof}[Proof that Proposition \ref{prop:ChenPrimesinBohrSetsModQ} implies Theorem~\ref{th:ChenPrimesinBohrSets}]
By Lemma~\ref{le:sumoverc0} it remains to show ~\eqref{eq:sumoverc_0inC} for $\mathcal{Q} = \mathcal{C}_C$ and $1 \neq q \mid Q$ with $\eta = w^{-1/2}$. This time we take $R(n) = (Wn + b)(Wn+b+2)$, and the claim follows exactly as in the previous proof, with $k = 2$.
\end{proof}

\section{Exponential sum estimates}
\label{sec:expsumest}
In this section we state exponential sum estimates that we will use in proofs of Propositions~\ref{prop:AlmMaynPrimesinBohrSetsModQ} and~\ref{prop:ChenPrimesinBohrSetsModQ}. Since the proofs closely follow previous works, we postpone them to Appendix~\ref{app:Expsums}.

\subsection{Major arc estimates}
\begin{lemma}
\label{le:MajorArcsPrimes}
Let $C_1, C_2 \geq 1$ and $\varepsilon > 0$. There exists a constant $x_0 = x_0(C_1, C_2, \varepsilon)$ such that the following holds. Let $Q \leq (\log x)^{C_1}$ and let $q \geq 1$ and $a$ be integers such that $q \mid Q$ and $(a, q) = 1$. Assume that $|\alpha - \frac{a}{q}| \leq (\log x)^{C_1}/x$. Then, for every $x \geq x_0$,
\[
\sum_{r \leq x^{1/2-\varepsilon}} \max_{(c, rQ) = 1} \Biggl| \sum_{\substack{x \leq p < 2x \\ p \equiv c \pmod{rQ}}} e\left(\alpha p \right) - \frac{Q}{\varphi(rQ)} \sum_{\substack{x \leq n < 2x \\ n \equiv c \pmod{Q}}} \frac{e(\alpha n)}{\log n}\Biggr| \leq \frac{x}{Q(\log x)^{C_2}}.
\]
\end{lemma}

\begin{lemma}
\label{le:MajorArcsTypeII}
Let $C_1, C_2 \geq 1$ and $\varepsilon > 0$. There exists a constant $x_0 = x_0(C_1, C_2, \varepsilon)$ such that the following holds. Let $Q \leq (\log x)^{C_1}$ and let $q \geq 1$ and $a$ be integers such that $q \mid Q$ and $(a, q) = 1$. Assume that $|\alpha - \frac{a}{q}| \leq (\log x)^{C_1}/x$. Then, for every $x \geq x_0$, any bounded sequences $\{a_m\}$ and $\{b_n\}$, and any $x^{1/4} \leq M \leq x^{3/4}$,
\[
\begin{split}
&\sum_{r \leq x^{1/2-\varepsilon}} \max_{(c,rQ) = 1} \Biggl| \sum_{\substack{x \leq mn < 2x \\ mn \equiv c \pmod{rQ} \\ M \leq m < 2M}} a_m b_n e\left(\alpha mn \right) - \frac{1}{\varphi(rQ)} \sum_{\substack{x \leq mn < 2x \\ (mn,rQ) = 1 \\ M \leq m < 2M}} a_m b_n e(\alpha mn) \Biggr| \\
&\leq \frac{x}{Q(\log x)^{C_2}}.
\end{split}
\]
\end{lemma}

\subsection{Minor arc estimates}
Our minor arc estimates are close variants of those proved in earlier papers. In particular we follow~\cite{MatomakiB-V} which in turn is based on ideas developed in~\cite{BalogPerelli, Mikawa}.

\begin{lemma}[Type I estimate]
\label{le:TypeISum}
There exists $x_0$ such that the following holds. Let $Q, q \geq 1$ and $a$ be integers such that $(a, q) = 1$. Let $|a_m| \leq 1$. Write $h = (q, Q)$. Assume that $\alpha$ is such that $|\alpha-a/q| < 1/(Qq^2)$ and that $Q \leq x^{1/2}$. Then, for every $x \geq x_0$ and any $M \geq 1$,
\[
\sum_{r \leq x^{1/2}} \max_{(c, rQ) = 1} \Biggl| \sum_{\substack{x \leq mn < 2x \\ mn \equiv c \pmod{rQ} \\ M \leq m < 2M}} a_m e\left(\alpha mn \right)\Biggr| \leq \frac{x}{Q} \left(\left(\frac{h}{q}\right)^{1/2} + \left(\frac{MQ}{x^{1/2}}\right)^{1/2} + \left(\frac{q}{x/Q}\right)^{1/2} \right) (\log x)^4. 
\]
\end{lemma}

\begin{lemma}[Type II estimate]
\label{le:TypeIIminorarc}
Let $C \geq 1$. There exists a constant $x_0 = x_0(C)$ such that the following holds. Let $Q, q \geq 1$ and $a$ be integers such that $(a, q) = 1$, write $h = (q, Q^2)$, and assume that $|\alpha - a/q| < 1/(4 q^2 Q^2 (\log x)^{2C})$. Let $M \in [x^{1/2}, x^{3/4}], Q \leq x^{3/2}/(2M^2 (\log x)^C)$, $D \leq x / (M Q (\log x)^C)$ and $R \leq M/x^{1/2}$, and let $c' \in \mathbb{Z}$.

Then, for every $x \geq x_0$ and any $|a_k|, |b_k| \leq \tau(k)$,
\[
\begin{split}
&\sum_{\substack{D \leq d < 2D}} \max_{(c, dQ) = 1} \sum_{\substack{R \leq r < 2R \\ (r,c'dQ) = 1}} \Biggl| \sum_{\substack{x \leq mn < 2x \\ mn \equiv c' \pmod{r} \\ mn \equiv c \pmod{dQ} \\ M \leq m < 2M}} a_m b_n e\left(\alpha mn \right)\Biggr| \\
&\leq \frac{x}{Q} \cdot \Bigl(\frac{(\log x)^{C/2}}{(q/h)^{1/8}} + (\log x)^{C/2} Q^{1/2} \frac{q^{1/8}}{x^{1/8}} + \frac{1}{(\log x)^{C/8}} \Bigr) (\log x)^{10}.
\end{split}
\]
\end{lemma}

Combining the type I and II estimates through Vaughan's identity we will obtain the following minor arc estimates for exponential sums over primes.
\begin{lemma}
\label{le:MinorArcMaynardprimes}
Let $C \geq 1$. There exists a constant $x_0 = x_0(C)$ such that the following holds. Let $Q, q \geq 1$ and $a$ be integers such that $(a, q) = 1$, write $h = (q, Q^2)$, and assume that $|\alpha - a/q| < 1/(4 q^2 Q^2 (\log x)^{2C})$. Then, for every $x \geq x_0$, and $Q \leq x^{1/10}$, 
\[
\begin{split}
&\sum_{r \leq x^{1/8}} \max_{(c, rQ) = 1} \Biggl| \sum_{\substack{x \leq p < 2x \\ p \equiv c \pmod{rQ}}} e\left(\alpha p \right)\Biggr| \leq \frac{x}{Q} \cdot \Bigl(\frac{(\log x)^{C/2}}{(q/h)^{1/8}} + (\log x)^{C/2} Q^{1/2} \frac{q^{1/8}}{x^{1/8}} + \frac{1}{(\log x)^{C/8}} \Bigr) (\log x)^{15}.
\end{split}
\]
\end{lemma}

\begin{lemma}
\label{le:MinorArcChenprimes}
Let $C \geq 1$. There exists a constant $x_0 = x_0(C)$ such that the following holds. Let $Q, q \geq 1$ and $a$ be integers such that $(a, q) = 1$, write $h = (q, Q^2)$, and assume that $|\alpha - a/q| < 1/(4 q^2 Q^2 (\log x)^{2C})$. Let $\lambda_r$ be as in Hypothesis~\ref{hyp:ChenHyp}(1).

Then, for every $x \geq x_0$, $Q \leq x^{\varepsilon/2}$, $(c, Q) = 1$ and $c' \in \mathbb{Z}$,
\[
\begin{split}
&\Biggl|\sum_{\substack{r \leq x^{1/2-\varepsilon} \\ (r,c'Q) = 1}} \mu(r)^2 \lambda_r \sum_{\substack{x \leq p < 2x \\ p \equiv c' \pmod{r} \\ p \equiv c\pmod{Q}}} e\left(\alpha p \right)\Biggr| \leq \frac{x}{Q} \cdot \Bigl(\frac{(\log x)^{C/2}}{(q/h)^{1/8}} + (\log x)^{C/2} Q^{1/2} \frac{q^{1/8}}{x^{1/8}} + \frac{1}{(\log x)^{C/8}} \Bigr) (\log x)^{15}.
\end{split}
\]
\end{lemma}

\section{Proof of Proposition~\ref{prop:AlmMaynPrimesinBohrSetsModQ}}
\label{sec:MaynApplication}
In this section we prove Proposition~\ref{prop:AlmMaynPrimesinBohrSetsModQ} using Maynard's Theorem (Theorem~\ref{th:Maynard}). Let us start by choosing the sequence $\omega_n$ and other parameters to which we apply Theorem~\ref{th:Maynard}. Let $C = C(1/8, 1/8)$ be as in Theorem~\ref{th:Maynard}, $k = \max\{C, e^{4Cm}\}$, and let $\rho = \rho(k, 1/8, 1/8)$ be as in Theorem~\ref{th:Maynard}. We take $x = N/Q$,
\[
(\omega_n) = (\chi(t - Q n - c_0)), \quad \text{and, for $i = 1, \dotsc, k$,} \quad L_i(n) = W(Q n+c_0+h_i)+b.
\]
We can assume that $\sum_{x \leq n < 2x} \omega_n \geq x/w^{10}$
since otherwise Proposition~\ref{prop:AlmMaynPrimesinBohrSetsModQ} is trivial. With these choices, we shall show that, for any $i = 1, \dotsc, k$,
\begin{equation}
\label{eq:MaynHyp1}
\sum_{r \leq x^{1/8}} \max_c \Biggl|\sum_{\substack{x \leq n < 2x \\ n \equiv c \pmod{r}}} \omega_n - \frac{1}{r} \sum_{\substack{x \leq n < 2x}} \omega_n\Biggr| \ll_{M, w} \frac{x}{(\log x)^{105k^2}},
\end{equation}
\begin{equation}
\label{eq:MaynHyp2}
\begin{split}
&\sum_{r \leq x^{1/8}} \max_{\substack{c \\ (W(Qc + c_0 + h_i)+b, r) = 1}} \Biggl|\sum_{\substack{x \leq n < 2x \\ n \equiv c \pmod{r} \\ W(Q n+c_0+h_i)+b \in \mathbb{P}}} \omega_n - \frac{QW}{\varphi(QWr)} \sum_{\substack{x \leq n < 2x}} \frac{\omega_n}{\log(W(Qn+c_0+h_i)+b)} \Biggr| \\
&\ll_{M, w} \frac{x}{(\log x)^{105k^2}},
\end{split}
\end{equation}
and that, for any $r \leq x^{1/8}$ and any $c$, we have
\begin{equation}
\label{eq:MaynHyp3}
\sum_{\substack{x \leq n < 2x \\ n \equiv c \pmod{r}}} \omega_n \ll_{M, w} \frac{x}{rw^{10}}.
\end{equation}

Now~\eqref{eq:MaynHyp1} implies Hypothesis~\ref{hyp:MaynardHyp}(1) and~\eqref{eq:MaynHyp3} implies Hypothesis~\ref{hyp:MaynardHyp}(3). Furthermore, looking only at the $r=1$ summand, we see that~\eqref{eq:MaynHyp2} implies that
\begin{equation}
\label{eq:MaynHyp2r=1}
\Biggl|\sum_{\substack{x \leq n < 2x \\ W(Q n+c_0+h_i)+b \in \mathbb{P}}} \omega_n - \frac{QW}{\varphi(QW)}   \sum_{\substack{x \leq n < 2x}} \frac{\omega_n}{\log(W(Qn+c_0+h_i)+b)} \Biggr| \ll_{M, w} \frac{x}{(\log x)^{105k^2}},
\end{equation}
which implies~\eqref{eq:MaynardHypExtra} with $\delta = 1/2$ (say). Furthermore, multiplying~\eqref{eq:MaynHyp2r=1} by $\varphi(QW)/\varphi(QWr)$ and summing over $r \leq x^{1/8}$, we see that
\[
\begin{split}
&\sum_{r \leq x^{1/8}} \max_{\substack{c \\ (W(Qc + c_0 + h_i)+b, r) = 1}} \Biggl|\frac{\varphi(QW)}{\varphi(QWr)} \sum_{\substack{x \leq n < 2x \\ W(Q n+c_0+h_i)+b \in \mathbb{P}}} \omega_n - \frac{QW}{\varphi(QWr)} \sum_{\substack{x \leq n < 2x}} \frac{\omega_n}{\log(W(Qn+c_0+h_i)+b)} \Biggr| \\
&\ll_{M, w} \frac{x}{(\log x)^{103k^2}},
\end{split}
\]
which together with~\eqref{eq:MaynHyp2} implies Hypothesis~\ref{hyp:MaynardHyp}(2) through the triangle inequality. 

Hence, assuming we can prove~\eqref{eq:MaynHyp1}--\eqref{eq:MaynHyp3}, recalling our choice of $k$, Maynard's theorem with $\delta = 1/2$ gives
\[
\sum_{\substack{x \leq n < 2x \\ \#(\{L_1(n), \dotsc, L_k(n)\} \cap \mathbb{P}) \geq m \\ p \mid L_1(n) \dotsm L_k(n) \implies p > x^{\rho}}} \omega_n \gg_m \frac{\mathfrak{S}(\mathcal{L})}{(\log x)^k} \sum_{x \leq n < 2x} \omega_n.
\]
Here
\[
\begin{split}
\mathfrak{S}(\mathcal{L}) &= \prod_p \left(1-\frac{\#\{1 \leq n \leq p \colon p \mid \prod_{i=1}^k (W(Q n+c_0+h_i)+b)\}}{p}\right)\left(1-\frac{1}{p}\right)^{-k} \\
&\gg \frac{1}{\exp(O(k))} \cdot \left(\frac{QW}{\varphi(QW)}\right)^k = \frac{1}{\exp(O(k))} \cdot \left(\frac{W}{\varphi(W)}\right)^k \cdot \frac{Q}{|\mathcal{C}_M|}
\end{split}
\]
by~\eqref{eq:|C|est}.

Recalling the definitions of $\omega_n$ and $L_i(n)$, we obtain,
\[
\begin{split}
\sum_{\substack{N \leq n < 2N \\ n \equiv c_0 \pmod{Q} \\ \#(\{W(n+h_i) + b\} \cap \mathbb{P}) \geq m \\ p \mid \prod_{i=1}^k (W(n+h_i) + b) \implies p \geq N^{\rho/2}}} \chi(t-n) &\gg_m \left(\frac{W}{\varphi(W)}\right)^k \cdot \frac{Q}{|\mathcal{C}_M|} \frac{1}{(\log x)^k} \sum_{\substack{N \leq n < 2N \\ n \equiv c_0 \pmod{Q}}} \chi(t-n)
\end{split}
\]
which was the claim.

Hence it remains to show~\eqref{eq:MaynHyp1}--\eqref{eq:MaynHyp3}. By the Fourier expansion of $\chi(n)$ in Lemma~\ref{le:chi-Fourier}, it is enough to show these with
\begin{equation}
\label{eq:omegaRedef}
\omega_n = e\left(\left(W \frac{a}{q} + \beta\right)Qn\right),
\end{equation}
where $0 \leq a < q \leq N/(\log N)^{100B}$, $(a,q) = 1$, $|\beta| \leq W(\log N)^{100B}/(qN)$, and, moreover, either $q \mid Q$ or $q/(q,Q^2) \geq (\log N)^A$. In particular~\eqref{eq:MaynHyp3} follows immediately from a trivial estimate.

We also note that when considering~\eqref{eq:MaynHyp1}--\eqref{eq:MaynHyp2} with $\omega_n$ as in~\eqref{eq:omegaRedef}, in case $|\beta| \leq 1/(Qx (\log x)^{111k^2})$ we can assume that $\beta = 0$ since $|e(y+h) - e(y)| = O(h)$. On the other hand if $|\beta| > 1/(Qx(\log x)^{111k^2})$, then this combined with the upper bound for $|\beta|$ implies that $|\beta| < 1/(4Q^2 q^2 (\log x)^{3200k^2})$. Hence we can in any case assume that 
\begin{equation}
\label{eq:betanewbound}
|\beta| < \min\left\{\frac{1}{4Q^2 q^2 (\log x)^{3200k^2}}, \frac{(\log x)^{110B}}{x}\right\}. 
\end{equation}

% and exponential sum estimates in Section~\ref{sec:expsumest}. % In order to apply the exponential sum estimates we sometimes need better control on the size of $\beta(s)$ than Fact~\ref{fa:chiFourierExp} gives. However note that we always have
% \[
% |\beta(s)| \leq \max \left\{\frac{1}{q(s)^2Q^2 W^3(\log x)^{5B}}, \frac{1}{Q W^3x} \right\}.
% \]
% In case $|\beta(s)| \leq 1/(QW^3x)$, we can use partial summation to reduce any exponential sum with $(W\frac{a(s)}{q(s)} + \beta(s))Qn$ into the same sum with $W\frac{a(s)}{q(s)} Qn$. Hence the hypotheses in Secton~\ref{sec:expsumest} concerning size of $|\alpha-a/q|$ can always be assumed to hold.

\subsection{Establishing~\eqref{eq:MaynHyp1}}
\label{ssec:MaynHyp1}
% With our choice of $\omega_n$, Hypothesis~\ref{hyp:MaynardHyp}(1) follows if we can show that
% \[
% \sum_{r \leq x^{1/8}} \max_a \left|\sum_{\substack{x \leq n < 2x \\ n \equiv a \pmod{r}}} \chi(t-Qn-c_0) - \frac{1}{r} \sum_{\substack{x \leq n < 2x}} \chi(t-Qn-c_0) \right| \ll_{M, w} \frac{x}{(\log x)^{110k^2}}.
% \]
% Recalling the Fourier expansion of $\chi(n)$ in Lemma~\ref{le:chi-Fourier}, it suffices to show that,
% \begin{equation}
% \label{eq:Hyp1claim}
% \begin{split}
% &\sum_{r \leq x^{1/8}} \max_a \left|\sum_{\substack{x \leq n < 2x \\ n \equiv a \pmod{r}}} e\left(\left(W\frac{a}{q}+\beta\right)Qn\right) - \frac{1}{r} \sum_{\substack{x \leq n < 2x}} e\left(\left(W\frac{a}{q}+\beta\right)Qn\right) \right| \\
% &\ll \frac{x}{(\log x)^{110k^2}},
% \end{split}
% \end{equation}

For $q \mid Q$, the left hand side of~\eqref{eq:MaynHyp1} with $\omega_n$ as in \eqref{eq:omegaRedef} equals
\[
\begin{split}
&\sum_{r \leq x^{1/8}} \max_c \Biggl|\sum_{\substack{x \leq n < 2x \\ n \equiv c \pmod{r}}} e\left(\beta Qn\right) - \frac{1}{r} \sum_{\substack{x \leq n < 2x}} e\left(\beta Qn\right) \Biggr| \ll \sum_{r \leq x^{1/8}} (|\beta| Q x + 1) \ll x^{1/2}
\end{split}
\]
by Lemma~\ref{le:GeomSerI}. 

For $q \nmid Q$, the left hand side of~\eqref{eq:MaynHyp1} with $\omega_n$ as in \eqref{eq:omegaRedef} is by triangle inequality at most
\begin{equation}
\label{eq:Hyp1minorarc}
\log x \sum_{r \leq x^{1/8}} \max_c \Biggl|\sum_{\substack{x \leq n < 2x \\ n \equiv c \pmod{r}}} e\left(\left(W\frac{a}{q}+\beta\right)Qn\right)\Biggr|.
\end{equation}
Recall~\eqref{eq:betanewbound} and that $q/(q, WQ) \geq (\log N)^{A}/W$, so that, by Lemma~\ref{le:TypeISum} with $M = Q = h = 1$ and $q/(q, QW)$ in place of $q$, we obtain that~\eqref{eq:Hyp1minorarc} is at most
\[
x \left(\frac{W^{1/2}}{(\log N)^{A/2}} + \frac{1}{x^{1/4}} + \frac{N^{1/2}}{x^{1/2}(\log N)^{50B}}\right) (\log x)^4 \ll \frac{x}{(\log x)^{110k^2}}
\]
once $A$ is large enough in terms of $k$.

\subsection{Establishing~\eqref{eq:MaynHyp2}}

By changes of variables $p, n' = W(Qn+c_0+h_i)+b$ and $c' = W(Qc+c_0+h_i) + b$, the left hand side of~\eqref{eq:MaynHyp2} with $\omega_n$ as in \eqref{eq:omegaRedef} is at most
\[
\begin{split}
&\sum_{r \leq x^{1/8}} \max_{\substack{(c', QWr) = 1}} \Biggl|\sum_{\substack{QWx \leq p < 2QWx \\ p \equiv c' \pmod{QWr}}} e\left(\left(\frac{a}{q}+\frac{\beta}{W}\right)p\right) \\
& \qquad \qquad - \frac{QW}{\varphi(QWr)} \sum_{\substack{QWx \leq n' < 2QWx \\ n' \equiv c' \pmod{QW}}} (\log n')^{-1} e\left(\left(\frac{a}{q}+\frac{\beta}{W}\right)n\right) \Biggr| + O(x^{1/2}).
\end{split}
\]
In case $q \mid Q$ this is $O(x/(\log x)^{200k^2})$ by Lemma~\ref{le:MajorArcsPrimes} recalling~\eqref{eq:betanewbound}.

In case $q \nmid Q$, note that  $q/(q, (QW)^2) > (\log N)^{A}/W^2$ and recall~\eqref{eq:betanewbound}. We use the triangle inequality and estimate the two terms corresponding to the two sums inside the absolute values separately.  The contribution corresponding to the sum over $n'$ can be satisfactorily estimated by Lemma~\ref{le:TypeISum} with $r = M = 1$ after partial summation. Furthermore Lemma~\ref{le:MinorArcMaynardprimes} with $C = 1600k^2$ implies
\[
\begin{split}
&\sum_{r \leq x^{1/8}} \max_{\substack{(c', QWr) = 1}} \Biggl|\sum_{\substack{QWx \leq p < 2QWx \\ p \equiv c' \pmod{QWr}}} e\left(\left(\frac{a}{q}+\frac{\beta}{W}\right)p\right)\Biggr|\\
&\leq x \cdot \Bigl(\frac{(\log x)^{800k^2}}{((\log N)^{A}/W^2)^{1/8}} + (\log x)^{800k^2} Q^{1/2} W \frac{q^{1/8}}{x^{1/8}} + \frac{1}{(\log x)^{200k^2}} \Bigr) (\log x)^{15} \ll \frac{x}{(\log x)^{150k^2}}
\end{split}
\]
when $A$ is large enough in terms of $k$.

\section{Proof of Proposition~\ref{prop:ChenPrimesinBohrSetsModQ}}
\label{sec:ChenApplication}
In this section we prove Proposition~\ref{prop:ChenPrimesinBohrSetsModQ} using Chen's Theorem (Theorem~\ref{thm:chen}). Let $\CL = \{L_1, L_2\}$ be the collection of two linear forms $L_1(n) = W(Qn + c_0) + b$ and $L_2(n) = W(Qn + c_0) + b + 2$, and note that
\[ \FS(\CL) \asymp \prod_{p \mid QW} \left(1 - \frac{1}{p}\right)^{-2} =  \left( \frac{QW}{\varphi(QW)} \right)^2 \asymp \frac{W^2}{\varphi(W)^2}  \frac{Q}{|\mathcal{C}_C|}. \]  
Let $x = N/Q$. Define the sequence $(\omega_n)$ for $x \leq n < 2x$ by
\[ \omega_n = \chi(t - Qn - c_0). \]
Since $\chi$ has Fourier complexity at most $M$, we have $\omega_n \leq M^2$ for every $n$. Thus the conclusion follows from Chen's theorem (Theorem~\ref{thm:chen}), once we verify the hypotheses. We may assume that $\sum_{x \leq n < 2x} \omega_n \geq x/(\log x)^{100}$ since otherwise the conclusion is trivial. Under this assumption, it suffices to show that, for $\lambda_r$ as in Hypothesis~\ref{hyp:ChenHyp}(1),
\begin{equation}\label{eq:ChenHyp1} 
\sum_{\substack{r \\ (r,QW)=1}} \mu(r)^2 \lambda_r \Biggl( \sum_{\substack{x \leq n < 2x \\ r \vert W(Qn + c_0) + b + 2 \\ W(Qn+c_0) + b \in \Prime}}  \omega_n - \frac{QW}{\varphi(QWr)} \sum_{\substack{x \leq n < 2x}} \frac{\omega_n}{\log (W(Qn+c_0) + b)} \Biggr) \ll \frac{x}{(\log x)^{200}}
\end{equation}
and that, for $B_j$ and $\lambda_r$ as in Hypothesis~\ref{hyp:ChenHyp}(2)
\begin{equation}\label{eq:ChenHyp2} 
\sum_{\substack{r \\ (r,QW)=1}} \mu(r)^2 \lambda_r \Biggl( \sum_{\substack{x \leq n < 2x \\ r \vert W(Qn + c_0) + b \\ W(Qn+c_0) + b + 2 \in B_j}}  \omega_n - \frac{\varphi(QW)}{\varphi(QWr)} \sum_{\substack{x \leq n < 2x \\ W(Qn+c_0)+b+2 \in B_j}} \omega_n \Biggr) \ll \frac{x}{(\log x)^{200}}
\end{equation}
and that, for $\delta(B_j)$ as in~\eqref{eq:density-Bj}
\begin{equation}\label{eq:ChenHyp3}
\sum_{\substack{x \leq n < 2x \\ W(Qn+c_0)+b+2 \in B_j}} \omega_n = \frac{\delta(B_j)+o(1)}{\log x} \cdot \frac{QW}{\varphi(QW)} \sum_{\substack{x \leq n < 2x}} \omega_n.
\end{equation}

By the Fourier expansion of $\chi(n)$ in Lemma~\ref{le:chi-Fourier}, it is enough to show these with
\begin{equation}
\label{eq:omegaRedef2}
\omega_n = e\left(\left(W \frac{a}{q} + \beta\right)Qn\right),
\end{equation}
where $0 \leq a < q \leq N/(\log N)^{100B}$, $(a,q) = 1$, $|\beta| \leq W(\log N)^{100B}/(qN)$, and, moreover, either $q \mid Q$ or $q/(q,Q^2) \geq (\log N)^A$. Furthermore, arguing as before (cf.~\eqref{eq:betanewbound}), we can assume
\begin{equation}
\label{eq:betanewbound2}
|\beta| < \min\left\{\frac{1}{4Q^2 q^2 (\log x)^{40 000}}, \frac{(\log x)^{110B}}{x}\right\}. 
\end{equation}
\subsection{Establishing \eqref{eq:ChenHyp1}}

After changes of variables $p, n' = W(Qn + c_0) + b$, we can rewrite the left hand side of~\eqref{eq:ChenHyp1} with $\omega_n$ as in \eqref{eq:omegaRedef2} essentially as
\[ 
\begin{split}
&\sum_{\substack{r \\ (r,QW) = 1}} \mu(r)^2 \lambda_r \Biggl( \sum_{\substack{QW x \leq p < 2 QWx \\ p \equiv -2\pmod{r} \\ p\equiv Wc_0 + b \pmod{QW} }}  e\left(\left(\frac{a}{q} + \frac{\beta}{W}\right) p \right) \\ &\qquad - \frac{QW}{\varphi(QWr)} \sum_{\substack{QWx \leq n' < 2QWx \\ n' \equiv Wc_0 + b \pmod{QW}}} (\log n')^{-1} e\left(\left(\frac{a}{q} + \frac{\beta}{W}\right) n' \right)\Biggr).
\end{split}
\]

In case $q \mid Q$, this is $O(x/(\log x)^{200})$ by Lemma~\ref{le:MajorArcsPrimes} recalling~\eqref{eq:betanewbound2}. In case $q \nmid Q$, note that  $q/(q, (QW)^2) > (\log N)^{A}/W^2$ and recall~\eqref{eq:betanewbound2}. We estimate the two terms corresponding to the sums over $p$ and $n'$ separately. The contribution from the term corresponding to the sum over $n'$ can be satisfactorily estimated by Lemma~\ref{le:TypeISum} with $r = M = 1$ after partial summation. For the term corresponding the sum over $p$, Lemma~\ref{le:MinorArcChenprimes} with $C = 20000$ implies the desired bound once $A$ and $B$ are large enough.

\subsection{Establishing \eqref{eq:ChenHyp2}}

By the definition of $B_1$ in \eqref{eq:ChenB} we can write
\[ \1_{W(Qn + c_0) + b + 2 \in B_1} = \sum_{\substack{mp = W(Qn + c_0) + b + 2 \\ p \geq x^{1/10}}} a_m, \]
where $a_m = 1$ if $m = p_1p_2$ for some $x^{1/10} \leq p_1 < x^{1/3-\varepsilon}$ and $x^{1/3-\varepsilon} \leq p_2 < (L_2(2x)/p_1)^{1/2}$, and $a_m = 0$ otherwise. Note that $a_m$ is supported on $m \in [x^{1/3}, x^{2/3}]$. After a dyadic division and changes of variables $mp = W(Qn + c_0) + b + 2$, to prove \eqref{eq:ChenHyp2} with $\omega_n$ as in~\eqref{eq:omegaRedef2} it suffices to show that for $M \in [x^{1/3}, x^{2/3}]$,
\[ 
\begin{split}
&\sum_{\substack{r \\ (r,QW) = 1}} \mu(r)^2 \lambda_r \Biggl( \sum_{\substack{QWx \leq mp < 2QWx \\ mp \equiv 2\pmod{r} \\ mp \equiv Wc_0 + b + 2\pmod{QW} \\ M \leq m < 2M }} a_m  e\left(\left(\frac{a}{q} + \frac{\beta}{W}\right) mp \right) \\
& \qquad - 
\frac{\varphi(QW)}{\varphi(QWr)} \sum_{\substack{QWx \leq mp < 2QWx \\ mp \equiv Wc_0+b+2\pmod{QW} \\ M \leq m < 2M }} a_m  e\left(\left(\frac{a}{q} + \frac{\beta}{W}\right) mp \right) \Biggr) \ll \frac{x}{(\log x)^{210}}.
\end{split}
\]

In case $q \mid Q$, this follows from Lemma~\ref{le:MajorArcsTypeII} applied twice (once with the $r=1$ term only), recalling~\eqref{eq:betanewbound2} and noting that we may add the restriction $(mp, QWr)=1$ in the second sum above at a negligible cost, since for each $r$ there are $O(x^{0.9})$ values of $mp$ with $(mp, QWr) > 1$. In case $q \nmid Q$, note that $q/(q, (QW)^2) > (\log N)^{A}/W^2$ and recall~\eqref{eq:betanewbound2}. We estimate the two sums separately.
The easier second sum can be estimated by Lemma~\ref{le:TypeISum} with $r = 1$. The first sum can be estimated by Lemma~\ref{le:TypeIIminorarc} (after factorizing $\lambda_r$) with $C = 20000$ once $A$ is large enough.

Hypothesis~\eqref{eq:ChenHyp2} for $B_2$ follows similarly noticing that
\[ \1_{W(Qn + c_0) + b + 2 \in B_2} = \sum_{\substack{mp = W(Qn + c_0) + b + 2 \\ p \geq x^{1/10}}} a_m, \]
where $a_m = 1$ if $m = p_1p_2$ for some $x^{1/3-\varepsilon} \leq p_1 \leq p_2 \leq (L_2(2x)/p_1)^{1/2}$ and $a_m = 0$ otherwise; thus $a_m$ is supported on $m \in [x^{2/3-2\varepsilon}, x^{2/3+o(1)}]$, so that our type II results (Lemmas~\ref{le:MajorArcsTypeII} and~\ref{le:TypeIIminorarc}) are still applicable. 

\subsection{Establishing \eqref{eq:ChenHyp3}}

In case $q \mid Q$, by partial summation it is enough to prove~\eqref{eq:ChenHyp3} in case $\beta = 0$ (strictly speaking one should consider the interval $n \in [x,x']$ instead of $n \in [x,2x)$ but this makes no difference). Since $q \mid Q$, we have $\omega_n \equiv 1$. By a change of variables $n' = W(Qn+c_0)+b+2$, it suffices to show that
\[  \sum_{\substack{QWx \leq n' < 2QWx \\ n' \equiv Wc_0+b+2\pmod{QW}}} \1_{n' \in B_j} = \frac{\delta(B_j)+o(1)}{\varphi(QW)} \cdot \frac{QWx}{\log x}, \]
which follows easily from the prime number theorem in arithmetic progressions. In case $q \nmid Q$, both sides of~\eqref{eq:ChenHyp3} are easily shown to be small using the argument from the previous subsection: the left hand side can be estimated by Lemma~\ref{le:TypeISum} and the right hand side can be estimated by Lemma~\ref{le:GeomSerII}.

\appendix

\section{Proof of generalized Chen's theorem}
\label{app:Chen}

In this section we prove Theorem~\ref{thm:chen}.

\subsection{The linear sieve}

For a (finitely supported) sequence $\CA = (a_m)$ of non-negative numbers we write $|\CA| = \sum_m a_m$ and $\CA_d = (a_{dm})_{m}$. We also define a sieving function
\[
S(\CA, z) = \sum_{(m, P(z)) = 1} a_m,
\]
where 
\[
P(z) = \prod_{p < z} p.
\]

In order to bound $S(\CA, z)$ we need some information about $\CA$. We will assume that, for all square-free integers $d$, we have
\[
|\CA_d| = \frac{g(d)}{d} X + r(\CA, d),
\]
where $g(d)$ is multiplicative and $X$ is independent of $d$. Let further
\[
V(z) = \prod_{p \mid P(z)} \left(1-\frac{g(p)}{p}\right).
\] 

We will use the linear sieve with a well-factorable error term due to Iwaniec \cite{Iw80}. For the following statement, see~\cite[Theorems 12.19 and 12.20]{Cribro}
\begin{lemma}
\label{le:lisieve}
Let $2 \leq z \leq D^{1/2}$ and $s = \log D/\log z$. Let $\varepsilon > 0$ be small enough and let $L(\varepsilon) = e^{1/\varepsilon^3}$. Assume that, for some absolute constant $K > 1$,
\[
\prod_{z_1 \leq p < z_2} \left(1-\frac{g(p)}{p}\right)^{-1} \leq K \frac{\log z_2}{\log z_1}
\]
for all $z_2 \geq z_1 \geq 2$. Then
\[
S(\CA, z) \leq X V(z) \left(F(s) + O_K(\varepsilon) \right) + \sum_{l < L(\varepsilon)} \sum_{d \mid P(z)} \lambda_l^+(d)r(\CA, d)
\]
and
\[
S(\CA, z) \geq X V(z) \left(f(s) - O_K(\varepsilon) \right) - \sum_{l < L(\varepsilon)} \sum_{d \mid P(z)} \lambda_l^-(d)r(\CA, d).
\]
Here, for each $l$, $\lambda_l^\pm$ are well-factorable functions of level $D$, and $F, f \colon [1, \infty) \to \mathbb{R}_{\geq 0}$ are the continuous solutions to the system
\[
\begin{cases}
sF(s) = 2e^\gamma & \text{if $1 \leq s \leq 3$;} \\
sf(s) = 0 &\text{if $1 \leq s \leq 2$;} \\
(sF(s))' = f(s-1) & \text{if $s > 3$;} \\
(sf(s))' = F(s-1) &\text{if $s > 2$}.
\end{cases}
\]
\end{lemma}

\subsection{Introducing Chen's weights}
Write $\CA = (a_m)$ for the sequence defined by
\[ a_m = \begin{cases} \omega_n \cdot \1_{L_1(n) \in \Prime} & m = L_2(n)\text{ for some } x \leq n < 2x \\ 0 & \text{otherwise.} \end{cases} \]
Note that $\CA$ is supported on $L_2(x) \leq m < L_2(2x)$.

Using a slight modification of the weighted sieve method of Chen, we consider
\[
\begin{split}
S = \sum_{\substack{m \\ (m,P(x^{1/10})) = 1}} a_m \Biggl(1 &- \frac{1}{2} \sum_{\substack{x^{1/10} \leq p_1 < x^{1/3-\varepsilon} \\ p_1 \mid m}} 1 \\
&- \frac{1}{2} \sum_{\substack{m = p_1 p_2 p_3 \\ x^{1/10} \leq p_1 < x^{1/3-\varepsilon} \\ x^{1/3-\varepsilon} \leq p_2 \leq (L_2(2x)/p_1)^{1/2} \\ p_3 \geq x^{1/10}}} 1 - \sum_{\substack{m = p_1 p_2 p_3 \\ x^{1/3-\varepsilon} \leq p_1 \leq p_2 \leq (L_2(2x)/p_1)^{1/2} \\ p_3 \geq x^{1/10}}} 1 \Biggr).
\end{split}
\]
Observe that the quantity in the parenthesis above is positive only if $m = P_2$ or $p^2 \mid m$ for some $x^{1/10} \leq p < x^{1/3-\varepsilon}$. Since the number of those $m$ of the latter type is $O(x^{0.9})$, it suffices to show that
\[ S \gg \frac{\FS(\CL)}{(\log x)^2} \sum_{x \leq n < 2x} \omega_n. \]
Using the sieve notation, we can write
\[
\begin{split}
S &= S(\CA, x^{1/10}) - \frac{1}{2} \sum_{x^{1/10} \leq p < x^{1/3-\varepsilon}} S(\CA_p, x^{1/10}) \\
& \qquad \qquad \qquad - \frac{1}{2} \sum_{\substack{p_1, p_2, p_3  \\ x^{1/10} \leq p_1 < x^{1/3-\varepsilon} \\ x^{1/3-\varepsilon} \leq p_2 \leq (L_2(2x)/p_1)^{1/2} \\ p_3 \geq x^{1/10}}} a_{p_1 p_2 p_3} - \sum_{\substack{p_1, p_2, p_3 \\ x^{1/3-\varepsilon} \leq p_1 \leq p_2 \leq (L_2(2x)/p_1)^{1/2} \\ p_3 \geq x^{1/10}}} a_{p_1 p_2 p_3}\\
&= S_1 - \frac{1}{2} S_2 - \frac{1}{2} T_1 - T_2,
\end{split}
\]
say.

\subsection{Handling $S_1$ and $S_2$}
Write 
\[
X = \frac{u_1}{\varphi(u_1)} \sum_{x \leq n < 2x} \frac{\omega_n}{\log L_1(n)}
\]
and let $g_1$ be the multiplicative function defined by
\[ 
g_1(d) = 
\begin{cases} 
0 & (d,u_2(u_2v_1 - u_1v_2)) > 1 \\ 
\frac{d\varphi(u_1)}{\varphi(u_1 d)} & (d,u_2(u_2v_1 - u_1v_2)) = 1. 
\end{cases} 
\]
Since $|\CA_d| = 0$ whenever $(d,u_2(u_2v_1 - u_1v_2)) > 1$, we have, by Hypothesis~\ref{hyp:ChenHyp},
\[ \sum_{d \mid P(x^{1/10})} \lambda_d \left( |\CA_d| - \frac{g_1(d)}{d} X \right) \ll (\log x)^{-10} \sum_{x \leq n < 2x} \omega_n. \]
for any well-factorable function $\lambda$ of level $D = x^{1/2-\varepsilon}$. 

Hence, by Lemma~\ref{le:lisieve} with $z = x^{1/10}$,
\[ S_1 \geq X V_1(x^{1/10}) (f(5-10\varepsilon) - o(1)) - O\left( (\log x)^{-9} \sum_{x \leq n < 2x} \omega_n \right), \]
where
\[ V_1(z) = \prod_{p \mid P(z)} \left(1 - \frac{g_1(p)}{p}\right) = \prod_{\substack{p < z \\ p\vert u_1,p\nmid u_2}} \left(1 - \frac{1}{p}\right) \prod_{\substack{p < z \\ p\nmid u_1u_2(u_1v_2 - u_2v_1)}} \left(1 - \frac{1}{p-1} \right). \]

Similarly, for any $2P \geq P' \geq P \in [x^{1/10}, x^{1/3-\varepsilon}]$ and any well-factorable bounded function $\lambda$ of level $x^{1/2-\varepsilon}/P$ we have, by Hypothesis~\ref{hyp:ChenHyp},
\[ 
\sum_{P \leq p < P'} \sum_{d \mid P(x^{1/10})} \lambda_d \left( |\CA_{pd}| - \frac{g_1(d)}{d} \frac{g_1(p)}{p} X \right) \ll (\log x)^{-10} \sum_{x \leq n < 2x} \omega_n, 
\]
since $|\CA_{pd}| = 0$ whenever $(d,u_2(u_2 v_1-u_1 v_2)) > 1$ and also $(p,d)=1$ whenever $d \vert P(x^{1/10})$.

By Lemma~\ref{le:lisieve} with $s = \log (x^{1/2-\varepsilon}/P)/ \log x^{1/10} = 5-10\varepsilon - 10\log P/\log x$, we obtain
\[ S_2 \leq \sum_{x^{1/10} \leq p < x^{1/3-\varepsilon}} \frac{g_1(p)}{p} X V_1(x^{1/10}) (F(5 -10\varepsilon- 10\log p/\log x) + o(1)) + O\left( (\log x)^{-9} \sum_{x\leq n < 2x} \omega_n \right). \]
Using the fact that
\[ X = \frac{u_1}{\varphi(u_1)} \cdot \frac{1+o(1)}{\log x} \sum_{x\leq n < 2x}\omega_n \]
since $\log L_1(n) = (1+o(1))\log L_1(x) = (1+o(1)) \log x$, we conclude that
\[ S_1 - \frac{1}{2} S_2 \geq  \frac{V(x^{1/10})}{\log x} \left( f(5-10\varepsilon) - \frac{1}{2}\int_{1/10}^{1/3-\varepsilon} F(5-10\varepsilon-10t) \frac{dt}{t} \right) (1-o(1)) \sum_{x\leq n < 2x} \omega_n, \]
where
\[ V(z) = V_1(z) \frac{u_1}{\varphi(u_1)} = \prod_{p\vert (u_1,u_2)} \frac{p}{p-1} \prod_{\substack{p \leq z \\ p \nmid u_1u_2(u_1v_2 - u_2v_1)}} \left(1 - \frac{1}{p-1} \right). \]

\subsection{Handling $T_1$ and $T_2$}
Let $j \in \{1, 2\}$. For $B_j$ defined as in~\eqref{eq:ChenB}, we write 
\[
X_j =  \sum_{\substack{x \leq n < 2x \\ L_2(n) \in B_j}} \omega_n
\]
and let $g_2$ be the multiplicative function defined by
\[ 
g_2(d) = 
\begin{cases} 
0 & (d,u_1(u_2v_1 - u_1v_2)) > 1 \\ 
\frac{d\varphi(u_2)}{\varphi(u_2 d)} & (d,u_1(u_2v_1 - u_1v_2)) = 1. 
\end{cases} 
\]
We consider the sequence $\CB^{(j)} = (b_m^{(j)})$  defined by
\[ b_m^{(j)} = \begin{cases} \omega_n \cdot \1_{L_2(n) \in B_j} & m = L_1(n) \text{ for some } x \leq n < 2x \\ 0 & \text{otherwise.} \end{cases} \]
Note that $\CB^{(j)}$ is supported on $L_1(x) \leq m < L_1(2x)$, and that, for $j = 1, 2$,
\[ T_j = \sum_{m \in \Prime} b_m^{(j)} \leq S(\CB^{(j)}, x^{1/6}). \] 
Note also that, for $j = 1, 2$,
\[ |\CB^{(j)}_d| = \sum_{\substack{x \leq n < 2x \\ d \vert L_1(n) }} \omega_n \1_{L_2(n) \in B_j}. \]
We may apply Hypothesis~\ref{hyp:ChenHyp}(2) to obtain that
\[ \sum_{d \mid P(x^{1/6})} \lambda_d \left( |\CB^{(j)}_d| - \frac{g_2(d)}{d} X_j \right) \ll (\log x)^{-10} \sum_{x \leq n < 2x} \omega_n \]
for any well-factorable function $\lambda_d$ of level $D = x^{1/2-\varepsilon}$.
Hence, by Lemma~\ref{le:lisieve} with $z = x^{1/6}$, we have
\[ T_j \leq X_j V_2(x^{1/6}) (F(3-6\varepsilon) + o(1)) + O \left( (\log x)^{-9} \sum_{x \leq n < 2x} \omega_n \right), \]
where
\[ V_2(z) = \prod_{\substack{p \leq z \\ p\vert u_2,p\nmid u_1}} \left(1 - \frac{1}{p}\right) \prod_{\substack{p \leq z \\ p\nmid u_1u_2(u_1v_2 - u_2v_1)}} \left(1 - \frac{1}{p-1} \right). \]
By Hypothesis~\ref{hyp:ChenHyp}(3) and using~\eqref{eq:density-Bj}, we have
\[ X_j \leq \frac{u_2}{\varphi(u_2)} \cdot \frac{\delta(B_j)+o(1)}{\log x} \sum_{x \leq n < 2x} \omega_n. \]
Hence
\[ T_j \leq \frac{V(x^{1/6})}{\log x} F(3-6\varepsilon) \delta(B_j) (1 + o(1)) \sum_{x \leq n < 2x} \omega_n, \]
since $V(z) = V_2(z) \frac{u_2}{\varphi(u_2)}$.

%We can handle $T_2$ similarly, obtaining
%\[ T_2 \leq \frac{V(x^{1/6})}{\log x} F(3-6\varepsilon) \int_{1/3-\varepsilon}^{1/3} \int_{\alpha_1}^{(1-\alpha_1)/2} \frac{d\alpha_2 d\alpha_1}{\alpha_1 \alpha_2 (1-\alpha_1-\alpha_2)} (1 + o(1)) \sum_{x \leq n < 2x} \omega_n. \]

\subsection{Final numerical work}

We may write
\[ V(z) = \Biggl( \prod_{p \vert (u_1,u_2)} \frac{p}{p-1} \prod_{\substack{p > 2 \\ p \vert u_1u_2(u_1v_2 - u_2v_1)}} \frac{p-1}{p-2} \Biggr) \prod_{2 < p \leq z} \left(1 - \frac{1}{p-1} \right), \]
and note that the two products in the parenthesis contribute $\gg \FS(\CL)$ by the definition of the singular series. Thus
\[ V(x^{1/6}) = \left(\frac{3}{5}+o(1)\right) V(x^{1/10}), \ \ V(x^{1/10}) \gg \frac{\FS(\CL)}{\log x}. \]
Since all the bounds we have obtained are continuous in $\varepsilon$ and the double integral in $\delta(B_2)$ from~\eqref{eq:density-Bj} tends to $0$ when $\varepsilon \to 0$, it suffices to verify that
\[ f(5) - \frac{1}{2} \int_{1/10}^{1/3} F(5-10t) \frac{dt}{t} - \frac{1}{2} \cdot \frac{3}{5} F(3) \int_{1/10}^{1/3} \int_{1/3}^{(1-\alpha_1)/2} \frac{d\alpha_2 d\alpha_1}{\alpha_1 \alpha_2 (1-\alpha_1-\alpha_2)} > 0 \]
just like in Chen's work. This is shown for instance in~\cite[Chapter 11]{HalbRichert}.

\section{Proof of the exponential sum estimates}
\label{app:Expsums}
In this appendix we prove a couple of very simple auxiliary lemmas as well as the exponential sum estimates stated in Section~\ref{sec:expsumest}.
\begin{lemma}
\label{le:GeomSerI}
Let $N \geq Q \geq 1$ and $c_0$ be integers, and let $\beta \in \R$. Then
\[
\sum_{\substack{N \leq n < 2N \\ n \equiv c_0 \pmod{Q}}} e(\beta n) = \frac{1}{Q} \sum_{\substack{N \leq n < 2N}} e(\beta n) +O(|\beta| N + 1).
\]
\end{lemma}
\begin{proof}
We can clearly assume that $0 \leq c_0 < Q$. Let us write $n = c_0 + kQ$, obtaining that 
\[
\begin{split}
\sum_{\substack{N \leq n < 2N \\ n \equiv c_0 \pmod{Q}}} e(\beta n) &= e(\beta c_0) \sum_{\substack{\frac{N-c_0}{Q} \leq k < \frac{2N-c_0}{Q}}} e(\beta k Q) = (1+O(\beta Q)) \sum_{\substack{\frac{N}{Q} \leq k < \frac{2N}{Q}}} e(\beta k Q) + O(1) \\
&= \sum_{\substack{\frac{N}{Q} \leq k < \frac{2N}{Q}}} e(\beta k Q) + O(|\beta| N + 1).
% = (1+O(\beta Q)) \frac{e\left(\beta \lfloor \frac{2N-c_0}{Q} \rfloor\right) - e\left(\beta \lceil\frac{N-c_0}{Q}\rceil\right)}{e(\beta)-1}. 
\end{split}
\]
Since the last expression is independent of $c_0$, summing over $0 \leq c_0 < Q$, we see that
\[
Q \sum_{\substack{\frac{N}{Q} \leq k < \frac{2N}{Q}}} e(\beta k Q) = \sum_{\substack{N \leq n < 2N}} e(\beta n) + O((|\beta| N + 1)Q),
\]
and the claim follows.
\end{proof}

\begin{lemma}
\label{le:GeomSerII}
Let $Q, q \geq 1$ and $a$ be integers such that $(a, q) = 1$ and $(Q, q) < q$. Assume that $|\alpha - a/q| \leq 1/(2qQ)$ and let $c_0 \in \mathbb{Z}$. Then
\[
\Biggl|\sum_{\substack{N \leq n < 2N \\ n \equiv c_0 \pmod{Q}}} e(\alpha n)\Biggr| \ll \frac{q}{(Q, q)}.
\]
\end{lemma}
\begin{proof}
Let us write $n = c_0 + kQ$, obtaining that
\[
\Biggl|\sum_{\substack{N \leq n < 2N \\ n \equiv c_0 \pmod{Q}}} e(\alpha n)\Biggr| = \Biggl|\sum_{\substack{\frac{N-c_0}{Q} \leq k < \frac{2N-c_0}{Q}}} e(\alpha k Q) \Biggr| \ll \frac{1}{\Vert \alpha Q\Vert} \leq \frac{1}{\frac{1}{2q/(Q, q)}}.
\]
\end{proof}

\subsection{Major arc estimates}

\begin{proof}[Proof of Lemma~\ref{le:MajorArcsPrimes}]
By partial summation it is enough to prove the claim in case $\alpha = a/q$ (strictly speaking one should consider intervals $p, n \in [x, x']$ instead of $[x, 2x]$ but this makes no difference).
Since $q \mid Q$, the left hand side of the claim equals
\[
\begin{split}
&\sum_{r \leq x^{1/2-\varepsilon}} \max_{(c, rQ) = 1} \Biggl| \sum_{\substack{x \leq p < 2x \\ p \equiv c \pmod{rQ}}} 1 - \frac{Q}{\varphi(rQ)}  \sum_{\substack{x \leq n < 2x \\ n \equiv c \pmod{Q}}} \frac{1}{\log n}\Biggr| \\
&\leq \sum_{r \leq x^{1/2-\varepsilon}} \max_{(c, rQ) = 1} \Biggl| \sum_{\substack{x \leq p < 2x \\ p \equiv c \pmod{rQ}}} 1 - \frac{|\mathbb{P} \cap [x, 2x)|}{\varphi(rQ)}\Biggr| + O(x(\log x)^{-C_1-C_2}) \\
&\leq \sum_{d \leq x^{1/2-\varepsilon/2}} \max_{(c, d) = 1} \Biggl| \sum_{\substack{x \leq p < 2x \\ p \equiv c \pmod{d}}} 1 - \frac{|\mathbb{P} \cap [x, 2x)|}{\varphi(d)}\Biggr| + O((\log x)^{-C_1-C_2}),
\end{split}
\]
and the claim follows from the Bombieri-Vinogradov prime number theorem.
\end{proof}

\begin{proof}[Proof of Lemma~\ref{le:MajorArcsTypeII}]
Arguing similarly, Lemma~\ref{le:MajorArcsTypeII} reduces to showing
\[
\sum_{d \leq x^{1/2-\varepsilon/2}} \max_{(c,d) = 1} \Biggl| \sum_{\substack{x \leq mn < 2x \\ mn \equiv c \pmod{d} \\ M \leq m < 2M}} a_m b_n - \frac{1}{\varphi(d)} \sum_{\substack{x \leq mn < 2x \\ (mn, d) = 1 \\ M \leq m < 2M}} a_m b_n \Biggr| \ll \frac{x}{(\log x)^{2C_1+C_2+1}}
\]
which follows from type II information used in the proof of the Bombieri-Vinogradov prime number theorem, see e.g.~\cite[Theorem 17.4]{IwKo04}.
\end{proof}

\subsection{Minor arc estimates for type I sums}
Notice that all the minor arc estimates are trivial if $q > x$, so that we can always assume that $q \leq x$. Lemma~\ref{le:TypeISum} follows easily from the following slight variant of a lemma usually used in type I estimates.
\begin{lemma}
\label{le:expsumwdiv}
Let $q \geq 1$ and $a$ be integers such that $(a, q) = 1$ and assume that $|\alpha - a/q| < 1/q^2$. For any $x \geq M \geq 1$ and any integer $k \geq 2$,
\[
\sum_{M \leq m < 2M} \tau_k(m) \min\left\{\frac{x}{M}, \frac{1}{\Vert \alpha m \Vert}\right\} \ll_k \left(\frac{x}{q^{1/2}} + x^{1/2}M^{1/2} + x^{1/2}q^{1/2}\right) (\log 3x)^{k^2/2}. 
\]
\end{lemma}
\begin{proof} 
By the Cauchy-Schwarz inequality
\[
\begin{split}
&\left(\sum_{M \leq m < 2M} \tau_k(m) \min\left\{\frac{x}{M}, \frac{1}{\Vert \alpha m \Vert}\right\}\right)^2\\
&\leq \left(\sum_{M \leq m < 2M} \tau_k(m)^2 \frac{x}{M} \right) \cdot \left(\sum_{M \leq m < 2M} \min\left\{\frac{x}{M}, \frac{1}{\Vert \alpha m \Vert}\right\}\right)\\
&\ll_k x (\log 3x)^{k^2-1} \cdot \left(\frac{x}{q} + M + q\right) (\log 3x). 
\end{split}
\]
by a standard ingredient in type I estimates (see e.g.~\cite[Formula before Lemma 13.7]{IwKo04}).
\end{proof}

\begin{proof}[Proof of Lemma~\ref{le:TypeISum}]
We can clearly assume that $M \leq x^{1/2}/Q$. Write $S$ for the left hand side of the claim, and write $\overline{m}$ for the inverse of $m \pmod{rQ}$. Then
\[
S \leq \sum_{r \leq x^{1/2}} \max_{(c, rQ) = 1} \sum_{\substack{M \leq m < 2M \\ (m, rQ) = 1}} \Biggl| \sum_{\substack{x/m \leq n < 2x/m \\ n \equiv c\overline{m} \pmod{rQ}}} e\left(\alpha mn \right)\Biggr|.
\]
Writing $n = c\overline{m} + k rQ$, with $k$ running over an interval with elements of size $x/(mrQ) \gg 1$, and summing the geometric series, we see that
\[
\begin{split}
S &\ll \sum_{r \leq x^{1/2}} \sum_{M \leq m < 2M} \min\left\{\frac{x}{mrQ}, \frac{1}{\Vert \alpha mrQ\Vert}\right\} \leq \sum_{d \leq 2Mx^{1/2}} \tau(d) \min\left\{\frac{x/Q}{d}, \frac{1}{\Vert (\alpha Q) d\Vert}\right\}.
\end{split}
\]

By our assumption on $\alpha$, we have $|\alpha Q - \frac{Qa/h}{q/h}| < 1/q^2 \leq 1/(q/h)^2$. Hence, after a dyadic division on $d$, Lemma~\ref{le:expsumwdiv} gives
\[
S \ll \left(\frac{x}{Q} \cdot \frac{1}{(q/h)^{1/2}} + \left(\frac{x}{Q}\right)^{1/2} (Mx^{1/2})^{1/2} + \left(\frac{x}{Q}\right)^{1/2} (q/h)^{1/2}\right) (\log x)^3. 
\]
\end{proof}

\subsection{Minor arc estimates for type II sums}
In proof of Lemma~\ref{le:TypeIIminorarc} we use the following auxiliary exponential sum estimate due to Mikawa~\cite{Mikawa}, in the proof of which one Fourier expands the $\min$-function on the left hand side and uses Weyl differencing.
\begin{lemma}
\label{le:sumfrsq}
Let $|\alpha-a/q| < 1/q^2$ for some $(a, q) = 1$. For $0 < M, J \leq x$, one has
\[
\begin{split}
M \sum_{M \leq m < 2M} \sum_{J \leq j < 2J} \tau_3(j) \min\left\{\frac{x}{m^2 j}, \frac{1}{\Vert{\alpha m^2 j}\Vert}\right\} \ll \left(M^2 J + x^{3/4}\left(\frac{x}{q} + \frac{x}{M} + q\right)^{1/4}\right) (\log x)^8.
\end{split}
\]
\end{lemma}

\begin{proof}[Proof of Lemma~\ref{le:TypeIIminorarc}]
Let us first note that in case $D \leq (\log x)^C$ we can combine $dr = d' \in [D', 4D']$ with $2D' = 2DR \leq 2(\log x)^C M/x^{1/2} \leq x/(MQ (\log x)^C)$ which is still at most the upper bound for $D$ in Lemma~\ref{le:TypeIIminorarc}. This allows us to assume that $R=1$ in case $D \leq (\log x)^C$; combining $dr = d'$ introduces at worst a divisor function $\tau(d')$, but the claim follows in any case if we can show the claimed upper bound for
%  Hence using Cauchy-Schwarz to remove the divisor function $\tau(d')$ coming from this combining (as in proof of Lemma~\ref{le:expsumwdiv}), it is enough to show that when either $D > (\log x)^C$, or $D \leq (\log x)^C$ and $R = 1$, one has
% \[
% \begin{split}
% &\sum_{\substack{D \leq d < 2D}} \max_{(c, dQ) = 1} \sum_{\substack{R \leq r < 2R \\ (r,c'dQ) = 1}} \left| \sum_{\substack{x \leq mn < 2x \\ mn \equiv c' \pmod{r} \\ mn \equiv c \pmod{dQ} \\ M \leq m < 2M}} a_m b_n e\left(\alpha mn \right)\right| \\
% &\leq \frac{x}{Q} \left(\frac{(\log x)^{2C}}{(q/h)^{1/8}} + (\log x)^{2C} Q \frac{q^{1/8}}{x^{1/8}} + \frac{1}{(\log x)^{C/8}} \right) (\log x)^5.
% \end{split}
% \]
% It suffices to consider
\[ I = \sum_{D \leq d < 2D}\tau(d) \sum_{\substack{R \leq r < 2R \\ (r, c'dQ) = 1}} \theta(d,r) \sum_{\substack{x \leq mn < 2x \\ mn \equiv c' \pmod{r} \\ mn \equiv c_d \pmod{dQ} \\ M \leq m < 2M}} a_m b_n e\left( \alpha mn \right), \]
for any choice of residue class $c_d \pmod{dQ}$ with $(c_d, dQ) = 1$ and any choice of $\theta(d,r) \in \C$ with $|\theta(d,r)| = 1$. By the Cauchy-Schwarz inequality, we have
\[ |I|^2 \ll DM (\log x)^6 \sum_{D \leq d < 2D} \sum_{M \leq m < 2M} \Biggl| \sum_{\substack{R \leq r < 2R \\ (r,c'dQ) = 1}} \theta(d,r) \sum_{\substack{x/m \leq n < 2x/m \\ mn \equiv c' \pmod{r} \\ mn \equiv c_d \pmod{dQ}}} b_n e(\alpha mn) \Biggr|^2. \]
Expanding out the square, moving the sum over $m$ inside, and noting that $|b_{n_1} b_{n_2}| \leq \tau(n_1)^2 + \tau(n_2)^2$, we obtain
\[ |I|^2 \ll DM (\log x)^6 \sum_{D \leq d < 2D} \sum_{\substack{R \leq r_1, r_2 < 2R \\ (r_1, c'dQ) = (r_2, c'dQ) = 1}} \sum_{\frac{x}{2M} \leq n_1, n_2 \leq \frac{2x}{M}} \tau(n_1)^2 \Biggl| \sum_{\substack{M \leq m < 2M \\ x/n_j \leq m < 2x/n_j \\ mn_1 \equiv c' \pmod{r_1} \\ mn_2 \equiv c' \pmod{r_2} \\ mn_1 \equiv mn_2 \equiv c_d \pmod{dQ}}} e(\alpha m (n_1 - n_2) ) \Biggr|. \]
The simultaneous congruences above are soluble if and only if $(n_1, r_1dQ) = (n_2, r_2dQ) = 1$ and $n_1 \equiv n_2 \pmod{dQ (r_1,r_2)}$, in which case they reduce to the single equation $m \equiv b \pmod{dQ [r_1,r_2]}$ for some $b$. Thus, substituting $m = b+kdQ[r_1, r_2]$ (and noticing $DQR^2 \leq M$), we see that the inner sum over $m$ is 
\[ \ll \min\left( \frac{M}{dQ [r_1,r_2]}, \frac{1}{\| \alpha(n_1-n_2) dQ [r_1,r_2] \|} \right). \]
Writing $n_1 = n_2 + \ell \cdot dQ(r_1,r_2)$, we have
\[ \alpha(n_1 - n_2) dQ [r_1, r_2] = \alpha \ell (dQ)^2 (r_1,r_2) [r_1,r_2] = \alpha \ell (dQ)^2 r_1r_2, \]
so that
\[ 
\begin{split}
|I|^2 &\ll DM (\log x)^6 \sum_{D \leq d < 2D} \sum_{R \leq r_1, r_2 < 2R} \sum_{\frac{x}{2M} \leq n_1 \leq \frac{2x}{M}} \tau(n_1)^2\sum_{|\ell| \leq \frac{2x}{MdQ(r_1,r_2)}} \min \left( \frac{M}{dQ [r_1,r_2]}, \frac{1}{\| \alpha \ell (dQ)^2 r_1r_2\|} \right) \\
&\ll Dx (\log x)^9 \sum_{D \leq d < 2D} \sum_{R \leq r_1, r_2 < 2R} \sum_{|\ell| \leq \frac{2x}{MdQ(r_1,r_2)}} \min \left( \frac{M}{dQ [r_1,r_2]}, \frac{1}{\| \alpha \ell (dQ)^2 r_1r_2\|} \right).
\end{split}
\]
The terms with $\ell = 0$ contribute to the right hand side
\[ \ll D^2 x (\log x)^9 \sum_{R \leq r_1, r_2 < 2R} \frac{M}{DQ [r_1,r_2]} \ll D^2 x (\log x)^{10} \frac{M}{DQ}  \ll \frac{x^2}{Q^2} (\log x)^{-C+10} \]
by the assumption on $D$, which is acceptable. To treat the terms with $\ell \neq 0$, write $j = \ell r_1r_2$ so that
\[ 0 < |j| < 4 R^2 \cdot \frac{2x}{MDQ (r_1, r_2)} \leq \frac{8R^2x}{MDQ} \]
and that
\[ \frac{M}{dQ [r_1,r_2]} = \frac{M (r_1,r_2)}{dQ r_1r_2} \ll \frac{M}{DQR^2} \cdot \frac{R^2 x}{MDQ |j|} \ll \frac{x}{(dQ)^2 |j|}. \]
It follows that
\[ |I|^2 \ll \frac{x^2}{Q^2} (\log x)^{-C+10} +  Dx (\log x)^9 \sum_{D \leq d < 2D} \sum_{0 < |j| \leq 8R^2x/(MDQ)} \tau_3(j) \min\left( \frac{x}{(dQ)^2 |j|}, \frac{1}{\| \alpha (dQ)^2 j\|} \right). \]
By a dyadic division, it suffices to show that
\begin{equation}\label{eq:typeII} 
\begin{split}
&D \sum_{D \leq d < 2D} \sum_{J \leq j < 2J} \tau_3(j) \min\left( \frac{x/Q^2}{d^2j}, \frac{1}{\| (\alpha Q)^2 d^2 j\|} \right) \\
&\ll \frac{x}{Q^2} \left( \frac{(\log x)^{C}}{(q/h)^{1/4}} + (\log x)^{C} Q \left(\frac{q}{x}\right)^{1/4} + \frac{1}{(\log x)^{C/4}} \right) (\log x)^8 
\end{split}
\end{equation}
for $1 \leq J \leq 4R^2x/(DMQ)$.
To prove this we divide into two cases depending whether $D$ is large or small. 

\begin{case}
First assume that $D > (\log x)^C$. Note that, by assumption,
\[ \Biggl| \alpha Q^2 - \frac{a Q^2/h}{q/h} \Biggr| \leq \frac{1}{4 q^2 (\log x)^{2C}} \leq \frac{1}{(q/h)^2}. \]
It is also easy to see that $D, J \leq x/Q^2$. We may thus apply Lemma~\ref{le:sumfrsq} to bound the left hand side of \eqref{eq:typeII} by
\[  
\begin{split}
&\ll \left(D^2J + \left(\frac{x}{Q^2}\right)^{3/4} \left(\frac{x/Q^2}{q/h} + \frac{x/Q^2}{D} + q \right)^{1/4}\right) (\log x)^8 \\
&\ll \frac{x}{Q^2} \left(\frac{1}{(\log x)^C} + \frac{1}{(q/h)^{1/4}} + \frac{1}{(\log x)^{C/4}} + Q^{1/2}\left(\frac{q}{x}\right)^{1/4}\right) (\log x)^8
\end{split}
\]
by the upper bound on $J$ and the assumptions on $D$ and $R$.
\end{case}

\begin{case}
Now assume that $D \leq (\log x)^C$. Recall that in this case we can assume that $R = 1$. In this case, for each fixed $D \leq d < 2D$ we have by assumption
\[ \Biggl| \alpha (dQ)^2 - \frac{ a (dQ)^2}{q} \Biggr| \leq \frac{1}{q^2}. \]
Moreover the denominator of the fraction $a(dQ)^2/q$ is at least $q/hD^2 \geq q/(h(\log x)^{2C})$ after reducing it to the reduced form. Applying Lemma~\ref{le:expsumwdiv} to the inner sum over $j$ in~\eqref{eq:typeII} (noticing that $J \leq x/(D^2Q^2)$), we may bound the left hand side of \eqref{eq:typeII} by
\[ 
\begin{split}
&\ll D^2 \left( \frac{x/(Q^2 D^2)}{\left(q/(h(\log x)^{2C})\right)^{1/2}} + \left(\frac{x}{Q^2 D^2}\right)^{1/2} \cdot \left(\frac{x}{DMQ}\right)^{1/2} + \left(\frac{x}{Q^2 D^2}\right)^{1/2} q^{1/2} \right) (\log x)^{9/2} \\
&\ll \frac{x}{Q^2} \left( \frac{(\log x)^C}{(q/h)^{1/2}} + \frac{1}{(\log x)^{C/2}} + Q\left(\frac{q}{x}\right)^{1/2} (\log x)^C \right) (\log x)^{9/2}
\end{split}
\]
by our assumptions on $D$ and $M$.
\end{case}
\end{proof}

\subsection{Minor arc estimates for sums over primes}
\begin{proof}[Proof of Lemma~\ref{le:MinorArcChenprimes}]
By partial summation it is enough to consider the claim of Lemma~\ref{le:MinorArcChenprimes} with 
\[
\sum_{\substack{x \leq p < 2x \\ p \equiv c' \pmod{r} \\ p \equiv c\pmod{Q}}} e\left(\alpha p \right) \quad \text{replaced by} \quad \sum_{\substack{x \leq n < 2x \\ n \equiv c' \pmod{r} \\ n \equiv c\pmod{Q}}} \Lambda(n) e\left(\alpha n \right)
\]
Then, by a dyadic splitting on $r$, Vaughan's identity (see \cite[Proposition 13.4]{IwKo04}) with $y = z = x^{2/3}$, and further partial summation, it is then enough to show that, for any $M \leq x^{1/3}$, $R \leq x^{1/2-\varepsilon/2}$ and any $|a_m| \leq 1$, one has the type I estimate
\begin{equation}
\label{eq:primesTypeI}
\begin{split}
&\Biggl|\sum_{\substack{R \leq r < 2R \\ (r,c'Q) = 1}} \mu(r)^2\lambda_r \sum_{\substack{x \leq mn < 2x \\ mn \equiv c' \pmod{r} \\ mn \equiv c\pmod{Q} \\ M \leq m < 2M}} a_m e\left(\alpha mn \right)\Biggr| \\
&\leq \frac{x}{Q} \cdot \left(\frac{(\log x)^{C/2}}{(q/h)^{1/8}} + (\log x)^{C/2} Q^{1/2} \frac{q^{1/8}}{x^{1/8}} 
+ \frac{1}{(\log x)^{C/8}} \right) (\log x)^{12}.
\end{split}
\end{equation}
and that, for any $x^{1/3} \leq M \leq x^{2/3}$, $R \leq x^{1/2-\varepsilon/2}$ and any $|a_k|, |b_k| \leq \tau(k)$, one has the type II estimate
\begin{equation}
\label{eq:primesTypeII}
\begin{split}
&\Biggl|\sum_{\substack{R \leq r < 2R \\ (r,c'Q) = 1}} \mu(r)^2\lambda_r \sum_{\substack{x \leq mn < 2x \\ mn \equiv c' \pmod{r} \\ mn \equiv c\pmod{Q} \\ M \leq m < 2M}} a_m b_n e\left(\alpha mn \right)\Biggr| \\
&\leq \frac{x}{Q} \cdot \left(\frac{(\log x)^{C/2}}{(q/h)^{1/8}} + (\log x)^{C/2} Q^{1/2} \frac{q^{1/8}}{x^{1/8}} + \frac{1}{(\log x)^{C/8}} \right) (\log x)^{12}.
\end{split}
\end{equation}

The estimate~\eqref{eq:primesTypeI} follows directly from Lemma~\ref{le:TypeISum}. On the other hand, to estimate~\eqref{eq:primesTypeII}, by symmetry we may assume that $M \geq x^{1/2}$, and we take $D = \min\{R, x/(MQ (\log x)^C)\}$ and $R' = R/D$. Note that $D \geq \min\{R, x^{1/3-\varepsilon}\}$ and thus for either possibility of $\lambda_r$ from Hypothesis~\ref{hyp:ChenHyp}(1), by the well-factorability property we always get for the left hand side of~\eqref{eq:primesTypeII} the upper bound
\[
\begin{split}
&\sum_{\substack{d \leq D \\ (d,c'Q) = 1}} \sum_{\substack{r' \leq R' \\ (r', c'dQ) = 1}}  \Biggl| \sum_{\substack{x \leq mn < 2x \\ mn \equiv c' \pmod{dr'} \\ mn \equiv c\pmod{Q} \\ M \leq m < 2M}} a_m b_n e\left(\alpha mn \right)\Biggr| \\
&\leq \sum_{\substack{d \leq D \\ (d,Q) = 1}} \max_{(c_0, dQ) = 1} \sum_{\substack{r' \leq R' \\ (r', c'dQ) = 1}}  \Biggl| \sum_{\substack{x \leq mn < 2x \\ mn \equiv c' \pmod{r'} \\ mn \equiv c_0 \pmod{dQ} \\ M \leq m < 2M}} a_m b_n e\left(\alpha mn \right)\Biggr|,
\end{split}
\]
and the claim follows from Lemma~\ref{le:TypeIIminorarc} after dividing the variables $d$ and $r'$ dyadically.
\end{proof}

Let us note that, in the previous proof, in order to apply Lemma~\ref{le:TypeIIminorarc} when $M$ is close to $x^{2/3}$, we needed to take $D$ to be slightly smaller than $x^{1/3}$. This is in contrast to what was claimed in~\cite[Remark 10]{MatomakiB-V}, but the caused mistake in the proof of~\cite[Theorem 2]{MatomakiB-V} could be easily fixed by using a slight modification of Chen's weights used here in Appendix~\ref{app:Chen}.

\begin{proof}[Proof of Lemma~\ref{le:MinorArcMaynardprimes}]
The proof is analogous to the proof of Lemma~\ref{le:MinorArcChenprimes} but, since $r \leq x^{1/8}$, after a dyadic division to $R \leq r < 2R$, we can always take $D = R$ when we apply Lemma~\ref{le:TypeIIminorarc}.
\end{proof}

\bibliographystyle{plain}
\bibliography{../biblio}

\end{document}